\documentclass[10pt, leqno]{amsart}
\usepackage{amsfonts,amssymb}
\usepackage{graphicx}
\usepackage{amsmath}
\usepackage{amsthm}
\usepackage{amsrefs}
\usepackage{mathtools}
\usepackage{accents}
\usepackage{qsymbols}
\usepackage{latexsym}
\usepackage{chngcntr}
\usepackage[noadjust]{cite}
\usepackage{paralist}
\usepackage{enumitem}
\usepackage{bm}
\usepackage{esint}
\usepackage{csquotes}
\usepackage{xcolor}

\newtheorem{theorem}{Theorem}[section]

\newtheorem{lemma}[theorem]{Lemma}
\newtheorem{proposition}[theorem]{Proposition}
\newtheorem{corollary}[theorem]{Corollary}
\theoremstyle{definition}
\newtheorem{definition}[theorem]{Definition}
\newtheorem{remark}[theorem]{Remark}

\newtheorem{agreement}{Agreement}
\newtheorem{assumption}[theorem]{Assumption}
\newtheorem{example}[theorem]{Example}

\usepackage[plainpages=false,pdfpagelabels,
citecolor=red]{hyperref}
\usepackage{xcolor}
\hypersetup{
 colorlinks,
 linkcolor={cyan!90!black},
 citecolor={magenta},
 urlcolor={green!40!black}
} %important to load after amsthm and before the rest

\newcommand{\IR}{\mathbb{R}}

\newcommand{\IN}{\mathbb{N}}
\newcommand{\IZ}{\mathbb{Z}}

\renewcommand{\S}{{\mathrm{S}}}
\newcommand{\loc}{\mathrm{loc}}

\newcommand{\WW}{\mathcal{W}}
\renewcommand{\L}{\mathrm{L}}

\newcommand{\C}{\mathrm{C}}
\newcommand{\W}{\mathrm{W}}

\newcommand{\B}{\mathrm{B}}
\newcommand{\Lip}{\mathrm{Lip}}
\renewcommand{\d}{\mathrm{d}}

\newcommand{\eps}{\varepsilon}

\newcommand{\abs}[1]{\lvert#1\rvert}
\newcommand{\cl}[1]{\overline{#1}}
\newcommand{\bd}{\partial}
\DeclareMathOperator{\supp}{supp}
\DeclareMathOperator{\dist}{d}
\DeclareMathOperator{\qhdist}{k}
\DeclareMathOperator{\diam}{diam}

\numberwithin{equation}{section}

\title[Extendability of functions with partially vanishing trace]{Extendability of functions with partially vanishing trace}
\author[S.~Bechtel]{Sebastian Bechtel}
\author[R.~M.~Brown]{Russell M.\@ Brown}
\author[R.~Haller-Dintelmann]{Robert Haller-Dintelmann}
\author[P.~Tolksdorf]{Patrick Tolksdorf}
\address{Fachbereich Mathematik, Technische Universit\"at Darmstadt, Schlossgartenstr. 7, 64289 Darmstadt, Germany}
\email{bechtel@mathematik.tu-darmstadt.de}
\email{haller@mathematik.tu-darmstadt.de}
\address{Department of Mathematics, University of Kentucky, Patterson office tower, Lexington, KY 40506-0027, United States of America}
\email{russell.brown@uky.edu}
\address{Institut f\"ur Mathematik, Johannes Gutenberg-Universit\"at Mainz, Staudingerweg 9, 55099 Mainz, Germany}
\email{tolksdorf@uni-mainz.de}
\thanks{The first-named and last-named authors were partially supported by \enquote{Studienstiftung des deutschen Volkes}, the second-named author was supported in part by a grant from the Simons Foundation \#422756, and the last-named author was partially supported by the project ANR INFAMIE (ANR-15-CE40-0011).}

\date{\today}
\begin{document}
\begin{abstract}
Let $\Omega \subseteq \IR^d$ be open and $D\subseteq \bd\Omega$ be a closed part of its boundary. Under very mild assumptions on $\Omega$, we construct a bounded Sobolev extension operator for the Sobolev space $\W^{k , p}_D (\Omega)$, $1 \leq p < \infty$, which consists of all functions in $\W^{k , p} (\Omega)$ that vanish in a suitable sense on $D$. In contrast to earlier work, this construction is global and \emph{not} using a localization argument, which allows to work with a boundary regularity that is sharp at the interface dividing $D$ and $\bd \Omega \setminus D$. Moreover, we provide homogeneous and local estimates for the extension operator. Also, we treat the case of Lipschitz function spaces with a vanishing trace condition on $D$.
\end{abstract}

\maketitle
%%%%%%%%%%%%%%%%%%%%%%%%%%%%%%%%%%%%%%%%%%%%%%%%%%%%%%%%%%%%%%%%%%%%%%%%%%%%%%%%%%%%%%%%%%%%%%%%%%%%%%%%%%%%%%%%%%%%%%%%%%%%%%%%%%%%%%%%%%%%%%%%%%%%%%%%%%%%%%%%%%%%

%%%%%%%%%%%%%%%%%%%%%%%%%%%%%%%%%%%%%%%%%%%%%%%%%%%%%%%%%%%%%%%%%%%%%%%%%%%%%%%%%%%%%%%%%%%%%%%%%%%%%%%%%%%%%%%%%%%%%%%%%%%%%%%%%%%%%%%%%%%%%%%%%%%%%%%%%%%%%%%%%%%%
\section{Introduction}
\label{Sec: Introduction}

\noindent Sobolev spaces $\W^{k,p}_D(\Omega)$ that contain functions that only vanish on a portion $D$ of the boundary of some given open set $\Omega \subseteq \IR^d$ play an eminent role in the study of the mixed problem for second-order elliptic operators, see for example~\cite{Auscher_Badr_Haller-Dintelmann_Rehberg, Brewster_M_M_M, Egert, Egert_Haller-Dintelmann_Rehberg, Egert_Haller-Dintelmann_Tolksdorf, Egert_Tolksdorf, terElst_Haller-Dintelmann_Rehberg_Tolksdorf, terElst_Rehberg, Haller-Dintelmann_Knees_Rehberg, Taylor_Kim_Brown, Tolksdorf}. In the study of these spaces, an extension operator is a crucial tool.

Early contributions to the history of Sobolev extension operators include the works of Stein~\cite[pp. 180--192]{Stein} and Calder\'{o}n~\cite{Calderon} on Lipschitz domains as well as the seminal paper of Jones~\cite{Jones} on $(\eps,\delta)$-domains. The latter mentioned work was later refined by Chua~\cite{Chua} and Rogers~\cite{Rogers}. Even though all these constructions aim at the full Sobolev space $\W^{1,p}(\Omega)$, they restrict to bounded extension operators on the space with vanishing trace on $D$ and the extensions preserve the trace condition on $D$ if a mild regularity assumption is imposed, see~\cite[Lem.~3.4]{BE}.

All these constructions rely on regularity assumptions for the full boundary of the underlying set. However, if we consider a (relatively) interior point of $D$, then it is possible to extend the function by zero around that point, so that a relaxation on the boundary regularity is feasible. This effect was exploited using localization techniques by several authors, see Brewster, Mitrea, Mitrea, and Mitrea~\cite{Brewster_M_M_M} for a very mature incarnation of this idea using local $(\eps,\delta)$-charts, and~\cite{Haller-Dintelmann_Knees_Rehberg} for a version using Lipschitz manifolds. We will present both frameworks in detail in Section~\ref{Sec: Comparison with other results and examples} and show that they are included in our setup.

One drawback of this method is that the regularity assumption for the Neumann boundary part $\bd \Omega \setminus D$ has to hold not merely on this boundary portion but in a neighbourhood of it, which in particular contains interior points of $D$. This forbids all kinds of cusps that are arbitrarily close to the interface between the Dirichlet and the Neumann boundary part.

In this work, we will introduce an $(\eps,\delta)$-condition that is adapted to the Dirichlet condition on $D$. To be more precise, we also connect nearby points in $\Omega$ by $\eps$-cigars, but these are with respect to the Neumann boundary part $\bd\Omega \setminus D$ and \emph{not} the full boundary $\bd \Omega$, which means that $\eps$-cigars may \enquote{leave} the domain across the Dirichlet part $D$ to some extent that is measured by a \emph{quasi-hyperbolic distance} condition. This allows to have certain inward and outward cusps arbitrarily close to the interface between the Dirichlet and Neumann parts, see Example~\ref{Ex: Sector} for an illustrative example. However, there are types of cusps that are particularly nasty and which are excluded from our setting by the aforementioned quasihyperbolic distance condition. In Example~\ref{Ex: Interior boundary cusp in zero} we show that in these kinds of configurations there cannot exist a bounded extension operator, which emphasizes that it is indeed necessary that we have incorporated some further restriction in our setup. A detailed description of our geometric framework will be given in Assumption~\ref{Ass: Epsilon-delta assumption}.

Next, we give a precise definition of what we mean by the term \emph{extension operator}, followed by our main result.

\begin{definition}
\label{Def: Extension operator}
	Call a linear mapping $E$ defined on $\L^1_{\loc} (\Omega)$ into the measurable functions on $\IR^d$ an \emph{extension operator} if it satisfies $Ef(x) = f(x)$ for almost every $x\in \Omega$ and for all $f \in \L^1_{\loc} (\Omega)$.
\end{definition}

\begin{theorem}
\label{Thm: Extension theorem for large radius}
Let $\Omega \subseteq \IR^d$ be open and let $D \subseteq \partial \Omega$ be closed. Assume that $\Omega$ and $D$ are subject to Assumption~\ref{Ass: Epsilon-delta assumption}. Moreover, fix an integer $k \geq 0$. Then there exists an extension operator $E$ such that for all $1 \leq p < \infty$ and $0\leq \ell \leq k$ one has that $E$ restricts to a bounded mapping from $\W^{\ell,p}_D(\Omega)$ to $\W^{\ell,p}_D(\IR^d)$.
The operator norms of $E$ only depend on $d$, $p$, $K$, $k$, $\eps$, $\delta$, and $\lambda$.
\end{theorem}

In addition, we will present some further improvements for the first-order case in Theorem~\ref{Thm: Lipschitz extension} and Corollary~\ref{Cor: Extension operator} which include the case of Lipschitz spaces and an enlargement of admissible geometries, as well as \emph{local} and \emph{homogeneous} estimates in Theorem~\ref{Thm: Local homogeneous estimates}.

\subsection*{Outline of the article}

First of all, we will present our geometric setting in Section~\ref{Sec: Function Spaces and Geometry} and will also give precise definitions for the relevant function spaces. A comparison with existing results and several examples and counterexamples are given in Section~\ref{Sec: Comparison with other results and examples}.

Then, we dive into the construction of the extension operator. Sections~\ref{Sec: Whitney decompositions} and~\ref{Sec: Cubes and chains} are all about cubes. In there, we will define collections of exterior and interior cubes coming from two different Whitney decompositions, and will explain how an exterior cube can be reflected \enquote{at the Neumann boundary} to obtain an associated interior cube. In contrast to Jones', not all small cubes in the Whitney decomposition of $\overline{\Omega}$ are exterior cubes. The treatment of Whitney cubes which are \enquote{almost} exterior cubes are the central deviation from Jones construction and are thus the heart of the matter in this article. These two sections are highly technical.

Eventually, we come to the actual crafting of the extension operator for Theorem~\ref{Thm: Extension theorem for large radius} in Section~\ref{Sec: The extension operator}. This section also contains results on (adapted) polynomials which are needed to define the extension operator via \enquote{reflection}. The proof of Theorem~\ref{Thm: Extension theorem for large radius} will be completed in Section~\ref{Sec: Conclusion of the proof}. Before that, we introduce an approximation scheme that yields more regular test functions for $\W^{k,p}_D(\Omega)$ in Section~\ref{Sec: Approximation smooth functions}. This additional regularity is crucial for Proposition~\ref{Prop: Lipschitz continuous  representative}.

Finally, we present some additional first-order theory in Section~\ref{Sec: First order}, followed by some short observations on locality and homogeneity in Section~\ref{Sec: Homogeneous estimates} which build on an observation made in Remark~\ref{Rem: Homogeneous estimates}.

\subsection*{Acknowledgements}

We would like to thank Juha Lehrb\"ack for drawing our attention towards the notion of quasihyperbolic distances.

\subsection*{Notation}

\noindent Throughout this article, the dimension $d \geq 2$ of the underlying
Euclidean space $\IR^d$ is fixed. Open balls around $x\in \IR^d$ of radius $r>0$ are denoted by $\B(x,r)$ and for the corresponding closed ball we write $\overline{\B}(x,r)$. The closure and
complement of a set $A \subseteq \IR^d$ are denoted by $\overline{A}$ and $A^c$. The Euclidean norm of a complex vector as well
as the Lebesgue measure of a measurable set in $\IR^d$ are denoted by
$\abs{~\cdot~}$.
If not otherwise mentioned, cubes are closed and axis-parallel.
We write $\mathcal{P}_m$ for the set of polynomials on $\IR^d$ of degree at most $m$.
The vector $\nabla^m f \coloneqq (\partial^\beta f)_{|\beta|=m}$ is introduced for an $m$-times (weakly) differentiable function. The letters $\alpha$ and $\beta$ are always supposed the mean multi-indices, possibly subject to further constraints.
The distance of two sets
$A , B \subseteq \IR^d$ is denoted by $\dist(A , B)$ and in the case $A = \{x\}$
the distance is abbreviated by $\dist(x , B)$. For $A\subseteq \IR^d$ and $t>0$ put $N_t(A) \coloneqq \{ x\in \IR^d \colon \dist(x,A) < t \}$. The diameter of an arbitrary subset of $\IR^d$ is denoted by
$\diam(\cdot)$. Finally, we follow the standard conventions that the infimum over the empty set is $\infty$ and that $1/\infty = 0$.

%%%%%%%%%%%%%%%%%%%%%%%%%%%%%%%%%%%%%%%%%%%%%%%%%%%%%%%%%%%%%%%%%%%%%%%%%%%%%%%%%%%%%%%%%%%%%%%%%%%%%%%%%%%%%%%%%%%%%%%%%%%%%%%%%%%%%%%%%%%%%%%%%%%%%%%%%%%%%%%%%%%%
\section{Geometry and Function Spaces}
\label{Sec: Function Spaces and Geometry}

\subsection{Geometry}
\label{Subsec: The geometry}

Let $\Xi \subseteq \IR^d$ be open. For two points $x , y \in \Xi$ their \textit{quasihyperbolic distance}, first introduced by Gehring and Palka~\cite{Gehring_Palka}, is given by
\begin{align*}
 \qhdist_{\Xi}(x , y) \coloneqq \inf_{\gamma} \int_{\gamma} \frac{1}{\dist(z , \partial \Xi)} \; \lvert \d z \rvert,
\end{align*}
where the infimum is taken over all rectifiable curves $\gamma$ in $\Xi$ joining $x$ with $y$. Notice that its value might be $\infty$. This is the case if there is no path connecting $x$ with $y$ in $\Xi$. The function $\qhdist_{\Xi}$ is called the \textit{quasihyperbolic metric}. If $\Xi^{\prime} \subseteq \Xi$ define
\begin{align*}
 \qhdist_{\Xi} (x , \Xi^{\prime}) \coloneqq \inf\{\qhdist_{\Xi} (x , y) : y \in \Xi^{\prime}\} \mathrlap{\qquad (x \in \Xi).}
\end{align*}
To construct the Sobolev extension operator in Theorem~\ref{Thm: Extension theorem for large radius}, we will rely on the following geometric assumption.

\begin{assumption}
\label{Ass: Epsilon-delta assumption}
Let $\Omega \subseteq \IR^d$ be open, $D \subseteq \partial \Omega$ be closed, and define $\Gamma \coloneqq \partial \Omega \setminus D$.
%%RMB
We assume that there exist $\eps \in (0,1]$, $\delta\in (0,\infty]$ and $K > 0$ 
%TODO habe hier explizit reingeschrieben dass delta=infty zulaessig ist
such that for all points $x , y \in \Omega$ with $\abs{x - y} < \delta$ there exists a rectifiable curve $\gamma$ that joins $x$ and $y$ and takes values in $\Xi \coloneqq \IR^d \setminus \overline{\Gamma}$ and satisfies
\begin{align}
 \mathrm{length}(\gamma) &\leq \eps^{-1} \abs{x - y}, \tag{LC}\label{Eq: Length condition}\\
 \dist(z , \Gamma) &\geq \eps \frac{\abs{x - z} \abs{y - z}}{\abs{x - y}} \mathrlap{\qquad(z \in \gamma),} \tag{CC}\label{Eq: Carrot condition} \\
 \qhdist_{\Xi}(z , \Omega) &\leq K \mathrlap{\qquad\qquad\qquad(z \in \gamma).} \tag{QHD}\label{Eq: Quasihyperbolic distance condition}
\end{align}
Furthermore, assume that there exists $\lambda > 0$ such that for each connected component $\Omega_m$ of $\Omega$ holds
\begin{align} \label{Eq: Diameter Condition}
	\partial \Omega_m \cap \Gamma \neq \emptyset \qquad \Longrightarrow \qquad \diam(\Omega_m)\geq \lambda \delta. \tag{DC}
\end{align}
\end{assumption}

\begin{remark}
\label{Rem: Properties QHD}
Let $(\Xi_m)_m$ denote the connected components of $\Xi$. From $\dist(z , \partial \Xi_m) = \dist(z , \partial \Xi)$ for $z \in \Xi_m$ follows directly that $\qhdist_{\Xi}(x, y) = \qhdist_{\Xi_m}(x, y)$ holds for all $x , y \in \Xi_m$. Note that $\partial \Xi = \cl{\Gamma}$ since $\cl{\Gamma} \subseteq \partial \Omega$ contains no interior points. Moreover, for $x \in \Xi_m$ and $y \in \Xi_n$ with $m \neq n$ one has $\qhdist_{\Xi}(x , y) = \infty$ since there is no connecting path between those points. Finally, $\qhdist_{\IR^d}(x , y) = 0$ holds for all $x , y \in \IR^d$ by the convention $1/\infty=0$.
\end{remark}

\begin{remark}
\label{Rem: For geometry}
\begin{enumerate}
\item Consider the pure Dirichlet case $D = \partial \Omega$. Then the curves are allowed to take values in all of $\IR^d$. In particular, we may connect points by a straight line, so that~\eqref{Eq: Length condition} is clearly satisfied. Condition~\eqref{Eq: Carrot condition} is void and also~\eqref{Eq: Quasihyperbolic distance condition} is trivially fulfilled, see Remark~\ref{Rem: Properties QHD}. Moreover, the diameter condition is always fulfilled since there are no connected components that intersect $\Gamma$. Consequently, if $D = \partial \Omega$, then Assumption~\ref{Ass: Epsilon-delta assumption} is fulfilled for any open set $\Omega$. 
\item Consider the pure Neumann case $D=\emptyset$ and fix $\eps,\delta$. The curve $\gamma$ can only connect points in the same connected component of $\Omega$. Thus, $\Omega$ is the union of at most countably many $(\eps , \delta)$-domains, whose pairwise distance is at least $\delta$ and whose diameters stay uniformly away from zero. In particular, if $\delta=\infty$, then $\Omega$ is connected and unbounded.
\item A similar condition on the diameter of connected components was introduced in~\cite[Sec.~2]{Brewster_M_M_M} in order to transfer Jones' construction of the Sobolev extension operator in~\cite{Jones} to disconnected sets. In the situation of Assumption~\ref{Ass: Epsilon-delta assumption} the positivity of the radius \emph{only} ensures that the connected components of $\Omega$ whose boundaries have a common point with $\Gamma$ do not become arbitrarily small. This is because our construction is global and not using a localization procedure.
We will present a thorough comparison with the geometry from~\cite{Brewster_M_M_M} in Section~\ref{Sec: Comparison with other results and examples}.
%If $\Omega$ has only finitely many such connected components, this condition is void.
\end{enumerate}
\end{remark}

\subsection{Function spaces}
\label{Subsec: Sobolev spaces}

Write $\W^{k,p}(\Omega)$ for the vector space of all $\L^p(\Omega)$ functions that have weak derivatives up to the non-negative integer order $k$ and which are again in $\L^p(\Omega)$. Equip $\W^{k,p}(\Omega)$ with the usual norm. Note that by Rademacher's theorem $\W^{1,\infty}(\IR^d)$ coincides with the space $\Lip(\IR^d)$ of Lipschitz continuous functions. A particular consequence is that (locally) Lipschitz continuous functions are weakly differentiable. We will exploit this fact in Section~\ref{Sec: Conclusion of the proof}. Note that on domains a mild geometric assumption is needed to ensure that $\W^{1,\infty}(\Omega)$ coincides with $\Lip(\Omega)$. This can be observed by considering $\Omega = \B(0,1) \setminus [0,1)$ as a counterexample.

\begin{definition}
\label{Def: CDinftyOmega}
	Let $\Omega \subseteq \IR^d$ be open and let $D \subseteq \overline{\Omega}$ be closed. Define the space of smooth functions on $\Omega$ which vanish in a neighborhood of $D$ by
\begin{align*}
	\C_D^\infty(\Omega) \coloneqq \bigl\{ f\in \C^\infty(\Omega) \colon \dist(\supp(f),D) > 0 \bigr\}.
\end{align*}
\end{definition}
Using this space of test functions, we define Sobolev functions vanishing on $D$. Note that we exclude the endpoint case $p=\infty$ in that definition. However, in the case $k=1$, we will work with a related space  in Section~\ref{Sec: First order}.
\begin{definition}
\label{Def: WkpDOmega}
	Let $\Omega \subseteq \IR^d$ be open and let $D \subseteq \overline{\Omega}$ be closed. For a non-negative integer $k$ and $p \in [1 , \infty)$  define the Sobolev space $\W^{k , p}_D
	(\Omega)$ as the closure of $\C_D^{\infty} (\Omega) \cap \W^{k,p}(\Omega)$ in $\W^{k , p}(\Omega)$.
\end{definition}
In Section~\ref{Sec: Approximation smooth functions} we will see that even the space $\C_D^\infty(\IR^d) \cap \W^{k,p}(\Omega)$ is dense in $\W^{k,p}_D(\Omega)$ as long as we assume the geometry from Assumption~\ref{Ass: Epsilon-delta assumption}; In fact, we will approximate by compactly supported $\C_D^\infty(\IR^d)$ functions, which are therefore in particular in $\W^{k,p}(\Omega)$.

\section{Comparison with other results and examples}
\label{Sec: Comparison with other results and examples}
\noindent This section is devoted to compare our result with existing results. The most general geometric setup to construct a Sobolev extension operator for the spaces $\W^{1 , p}_D (\Omega)$ was considered in the work of Brewster, Mitrea, Mitrea, and Mitrea~\cite[Thm.~1.3, Def.~3.4]{Brewster_M_M_M} and reads as follows. 

\begin{assumption}
\label{Ass: Locally eps-delta near Gamma}
Let $\Omega \subseteq \IR^d$ be an open, non-empty, and proper subset of $\IR^d$, $D \subseteq \partial \Omega$ be closed, and let $\Gamma \coloneqq \partial \Omega \setminus D$. Let $\eps , \delta > 0$ be fixed. Assume there exist $r_0 > 0$ and an at most countable family $\{ O_j \}_j$ of open subsets of $\IR^d$ satisfying
\begin{enumerate}
 \item $\{ O_j \}_j$ is locally finite and has bounded overlap,
 \item \label{Item: Second BMMM assumption} for all $j$ there exists an $(\eps , \delta)$-domain $U_j \subseteq \IR^d$ with connected components all of diameter at least $r_0$ and satisfying $O_j \cap \Omega = O_j \cap U_j$,
 \item \label{Item: Third BMMM assumption}there exists $r \in (0 , \infty]$ such that for all $x \in \Gamma$ there exists $j$ for which $\B(x , r) \subseteq O_j$.
\end{enumerate}
\end{assumption}

Here, an open set $U_j$ is called an $(\eps , \delta)$-domain if there exist $\eps , \delta > 0$ such that for all $x , y \in U_j$ there exists a rectifiable curve $\gamma$ that joins $x$ and $y$, takes its values in $U_j$, and satisfies~\eqref{Eq: Length condition} and~\eqref{Eq: Carrot condition} with respect to $\partial U_j$ instead of $\Gamma$. Also note that~\eqref{Eq: Length condition} enforces $\eps \in (0,1]$.

\begin{proposition}
Assumption~\ref{Ass: Locally eps-delta near Gamma} implies Assumption~\ref{Ass: Epsilon-delta assumption}.
\end{proposition}

\begin{proof}
Let $\eps , \delta , r$, and $r_0$ be the quantities from Assumption~\ref{Ass: Locally eps-delta near Gamma}. We have to show the quantitative connectedness condition contained in~\eqref{Eq: Length condition} to~\eqref{Eq: Quasihyperbolic distance condition} as well as the diameter condition~\eqref{Eq: Diameter Condition} for connected components touching $\Gamma$. For the rest of the proof, references to~\eqref{Item: Second BMMM assumption} and~\eqref{Item: Third BMMM assumption} refer to the respective items in Assumption~\ref{Ass: Locally eps-delta near Gamma}.

For the first task, let $x,y\in \Omega$ and define $\kappa \coloneqq r / 8$ and $V_\kappa \coloneqq \{ x \in \IR^d \colon \dist(x , \Gamma) \leq \kappa \}$.
%Recall that $\eps \in (0 , 1]$ by Remark~\ref{Rem: For geometry}~\eqref{Enum: Size of epsilon} so that in particular $\delta^{\prime} < r / 8$.
We proceed by distinguishing two cases.

\emph{Case 1}: $x , y \in V_{\kappa}$ with $|x-y| < \min(\delta , \eps r / 8)$. Fix $x_0 \in \overline{\Gamma} \subseteq \bd \Omega$ such that $\dist(x , \Gamma) = |x-x_0|$. Then $x,y \in \B(x_0 , r/4)$ and by~\eqref{Item: Third BMMM assumption} we get $\B(x_0 , r) \subseteq O_j$ for some $j$. Using this and~\eqref{Item: Second BMMM assumption} we furthermore get $x,y \in \B(x_0 , r) \cap \Omega = \B(x_0 , r) \cap U_j$. Notice that this gives in particular $x_0 \in \partial U_j$. Next, let $\gamma$ denote the $(\eps , \delta)$-path subject to~\eqref{Eq: Length condition} and~\eqref{Eq: Carrot condition} that connects $x$ and $y$ in $U_j$. For $z \in \gamma$ we have by~\eqref{Eq: Length condition}
\begin{align*}
 \lvert x_0 - z \rvert \leq \lvert x_0 - x \rvert + \lvert x - z \rvert \leq \kappa + \mathrm{length} (\gamma) \leq \frac{r}{8} + \frac{|x-y|}{\eps} < \frac{r}{4}.
\end{align*}
Thus, $\gamma$ takes its values in $\B(x_0 , r / 4) \cap U_j = \B(x_0 , r / 4) \cap \Omega \subseteq \Xi$. This shows that $\gamma$ is an admissible path for Assumption~\ref{Ass: Epsilon-delta assumption}, and of course~\eqref{Eq: Length condition} stays valid. We also conclude
\begin{align*}
\dist(z, \bd \Omega)
= \dist\bigl(z, \B(x_0, r/2) \cap \bd\Omega \bigr)
= \dist\bigl(z, \B(x_0, r/2) \cap \bd U_j \bigr)
= \dist(z , \partial U_j),
\end{align*}
and thus, taking~\eqref{Eq: Carrot condition} with respect to $\bd U_j$ into account, we derive
\begin{align*}
 \dist(z , \Gamma) \geq \dist(z , \partial \Omega) = \dist(z , \partial U_j) \geq \eps \frac{\lvert x - y \rvert \lvert y - z \rvert}{\lvert x - y \rvert},
\end{align*}
which is~\eqref{Eq: Carrot condition} with respect to $\Gamma$. Finally, since $\gamma$ takes its values in $\Omega$, it satisfies~\eqref{Eq: Quasihyperbolic distance condition} with $K = 0$.

\emph{Case 2}: $x \in V_\kappa^c$ and $|x-y| < \kappa/2$ (the case with $y\in V_\kappa^c$ works symmetrically). Write $\gamma$ for the straight line that connects $x$ with $y$ in $\IR^d$. Then~\eqref{Eq: Length condition} is clearly fulfilled and for~\eqref{Eq: Carrot condition} we estimate with $z \in \gamma$ using $\max(|x-z|,|y-z|) \leq |x-y| \leq \kappa/2$ that
\begin{align*}
 \eps \frac{|x - z| |y - z|}{|x - y|} < \frac{\kappa}{2} < \dist(x, \Gamma) - \mathrm{length}(\gamma) \leq \dist(z, \Gamma).
\end{align*}
In particular, $\gamma$ takes its values in $\Xi$.

We continue with the second task. To control the quasihyperbolic distance of a point $z\in \gamma$ to $\Omega$ with respect to $\Xi$, we estimate $\qhdist_\Xi(z,x)$ using the line segment that connects $x\in \Omega$ with $z$. Then the integrand in the definition of the quasihyperbolic distance is bounded by $2/\kappa$ and the length of the path is at most $\kappa/2$. Hence, $\qhdist_\Xi(z,\Omega) \leq 1$, which gives~\eqref{Eq: Quasihyperbolic distance condition}.

Therefore, conditions~\eqref{Eq: Length condition} to~\eqref{Eq: Quasihyperbolic distance condition} are satisfied for all $x,y\in \Omega$ as long as $|x-y| < \min(\delta, \eps r/8, \kappa/2)=:\delta'$, which concludes the first part of this proof.

To show the diameter condition, let $\Omega_m$ be a connected component of $\Omega$ with $\bd \Omega_m \cap \Gamma \neq \emptyset$. Fix some $x_0$ in this intersection. We show $\diam(\Omega_m) \geq \min(r/2,r_0)$. This implies in particular that $\diam(\Omega_m) \geq \lambda \delta'$ for some suitable $\lambda$ since $\delta'$ is finite. Suppose $\diam(\Omega_m) < r/2$. Then $\Omega_m \subseteq \B(x_0,r/2)$. According to~\eqref{Item: Third BMMM assumption}, we have $\B(x_0,r) \subseteq O_j$ for some $j$. Taking also~\eqref{Item: Second BMMM assumption} into account, we get on the one hand that $\Omega_m \subseteq U_j$, and on the other hand that points in $U_j$ close to $\Omega_m$ belong to $\Omega$. We draw two conclusions: First, $\Omega_m$ is an open and connected subset of $U_j$. Second, if there were a continuous path in $U_j$ connecting a point from $U_j \setminus \Omega_m$ with one in $\Omega_m$, that path would eventually run in $\Omega$ and therefore $\Omega_m$ wouldn't be maximally connected in $\Omega$, leading to a contradiction. So in total, $\Omega_m$ is a connected component of $U_j$ and hence $\diam(\Omega_m) \geq r_0$.
\end{proof}

A common geometric setup which is used in many works dealing with mixed Dirichlet/Neumann boundary conditions, see for example~\cite{Auscher_Badr_Haller-Dintelmann_Rehberg, Egert, Egert_Haller-Dintelmann_Rehberg, Egert_Haller-Dintelmann_Tolksdorf, terElst_Haller-Dintelmann_Rehberg_Tolksdorf, terElst_Rehberg, Taylor_Kim_Brown, Tolksdorf}, requires Lipschitz charts around points on the closure of $\Gamma$ and is presented in the following assumption.

\begin{assumption}
\label{Ass: Lipschitz}
Let $\Omega \subseteq \IR^d$ be a bounded open set and $D \subseteq \partial \Omega$ be closed. Assume that around each point $x \in \overline{\Gamma}$ there exists a neighborhood $U_x$ of $x$ and a bi-Lipschitz homeomorphism $\Phi_x : U_x \to (-1 , 1)^d$ such that $\Phi_x (x) = 0$, $\Phi_x (U_x \cap \Omega) = (-1 , 1)^{d - 1} \times (0 , 1)$, and $\Phi_x (U_x \cap \bd\Omega) = (-1 , 1)^{d - 1} \times \{ 0 \}$.
\end{assumption}

\begin{proposition}
\label{Prop: Lipschitz}
Assumption~\ref{Ass: Lipschitz} implies Assumption~\ref{Ass: Locally eps-delta near Gamma}.
\end{proposition}

\begin{proof}
By~\cite[Lem.~2.2.20]{Egert_Dissertation}, for any $x \in \overline{\Gamma}$ the set $U_x \coloneqq O_x \cap \Omega$ is an $(\eps , \delta)$-domain. Here, $\eps$ and $\delta$ do only depend on $d$ and the Lipschitz constant. The compactness of $\overline{\Gamma}$ implies that there exist finitely many $x_1 , \dots , x_m \in \overline{\Gamma}$ such that $\overline{\Gamma} \subseteq \bigcup_{j = 1}^m O_{x_j}$. Define $O_j \coloneqq O_{x_j}$ and $U_j \coloneqq U_{x_j}$ for $j = 1 , \dots , m$. Due to the finiteness of the family $\{ U_j \}_{j = 1}^m$, the constants $\eps$ and $\delta$ can be chosen to be uniform in $j$. Finally, if $r > 0$ is the Lebesgue number of the covering $\{O_j\}_{j=1}^m$, then for all $x_0 \in \overline{\Gamma}$ there exists $1 \leq j \leq m$ such that $\B(x_0 , r) \subseteq O_j$. Thus, all requirements in Assumption~\ref{Ass: Locally eps-delta near Gamma} are fulfilled.
\end{proof}

Next, we give an example of a two-dimensional domain that satisfies Assumption~\ref{Ass: Epsilon-delta assumption} but not Assumption~\ref{Ass: Locally eps-delta near Gamma}. We further show that, within this configuration, the geometry described in Assumption~\ref{Ass: Epsilon-delta assumption} is in some sense optimal.
\begin{example}
\label{Ex: Sector}
Let $\theta \in (0 , \pi)$ and let $\S_{\theta} \subseteq \IR^2$ denote the open sector symmetric around the positive $x$-axis with opening angle $2 \theta$. Let $\Omega \subseteq \IR^2$ be any domain satisfying
\begin{align*}
 \Omega \cap \S_{\theta} = \{ (x , y) \in \S_{\theta}  : y < 0 \},
\end{align*}
and define
\begin{align*}
 D \coloneqq \partial \Omega \cap \bigl[\,\IR^2 \setminus \S_{\theta}\,\bigr] \qquad \text{and} \qquad \Gamma \coloneqq \partial \Omega \setminus D = (0 , \infty) \times \{ 0 \}.
\end{align*}
Essentially, this means that inside the sector $\S_{\theta}$ the domain $\Omega$ looks like the lower half-space and the half-space boundary that lies inside $\S_{\theta}$ is $\Gamma$. In the complement of the sector $\S_{\theta}$, $\Omega$ could be any open set and the boundary of $\Omega$ in the complement of $\S_{\theta}$ is defined to be $D$. See Figure~\ref{Fig: Wild geometry} for an example of such a configuration. \par
    \begin{figure}
	\centering
  \includegraphics[width=0.55\textwidth]{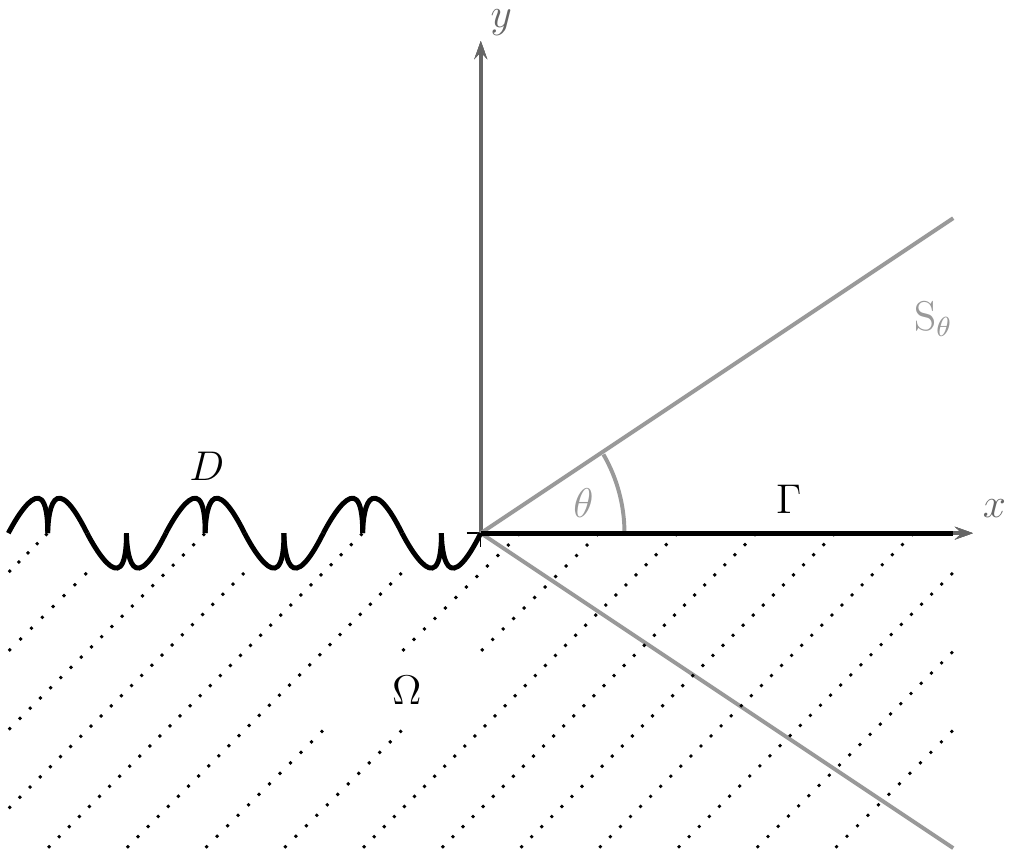}
  \caption{A generic picture of a domain described in Example~\ref{Ex: Sector}.}
  \label{Fig: Wild geometry}
    \end{figure}
To verify that such a domain fulfills the geometric setup described in Assumption~\ref{Ass: Epsilon-delta assumption}, consider first the set
\begin{align*}
 \Delta_{\theta} \coloneqq (\IR^2 \setminus \overline{\S_\theta} ) \cup \{ (x , y) \in \IR^2 : y < 0 \},
\end{align*}
which is an $(\eps , \delta)$-domain for some values $\eps , \delta > 0$. Since $\Omega \subseteq \Delta_\theta$ and $\overline{\Gamma} \subseteq \bd\Delta_\theta$, the $(\eps,\delta)$-paths with respect to $\Delta_\theta$ for points in $\Omega$ satisfy~\eqref{Eq: Length condition} and~\eqref{Eq: Carrot condition}. Hence, to conclude the example, we only have to show that there exists $K > 0$ such that for all $z \in \Delta_{\theta}$ and with $\Xi = \IR^2 \setminus \overline{\Gamma}$ it holds 
\begin{align}
\label{Eq: Distance in example}
 \qhdist_{\Xi} (\Omega , z) \leq K.
\end{align}
Since the paths obtained above take their values only in $\Delta_{\theta}$ this will establish the remaining condition~\eqref{Eq: Quasihyperbolic distance condition}. Notice that since $\S_{\theta} \cap \{ (x , y) \in \IR^2 : y < 0 \} \subseteq \Omega$ it suffices to show that there exists $K > 0$ such that for all $z \in \Delta_{\theta}$ it holds
\begin{align*}
 \qhdist_{\Xi} (\S_{\theta} \cap \{ (x , y) \in \IR^2 : y < 0 \} , z) \leq K.
\end{align*}
We only describe one particular case in detail, the remaining cases are similar and left to the interested reader. Assume that $\theta < \pi / 2$ and pick $z = (v , w) \in \Delta_{\theta}$ with $v \geq 0$ and $w > 0$. Choose $(x , y) \in \partial \S_{\theta}$ such that $y \coloneqq - w$ and let $\gamma \coloneqq \gamma_1 + \gamma_2 + \gamma_3$ with
\begin{align*}
 \gamma_1 &: [0 , 1] \to \IR^2, \quad t \mapsto (x , y) + t (y - x , 0), \\
 \gamma_2 &: [0 , 1] \to \IR^2, \quad t \mapsto (y , y) + t (0 , w - y), \\
 \gamma_3 &: [0 , 1] \to \IR^2, \quad t \mapsto (y , w) + t (v - y , 0).
\end{align*}
This construction is depicted in Figure~\ref{Fig: qhdistance}.
    \begin{figure}[!ht]
	\centering
  \includegraphics[width=0.43\textwidth]{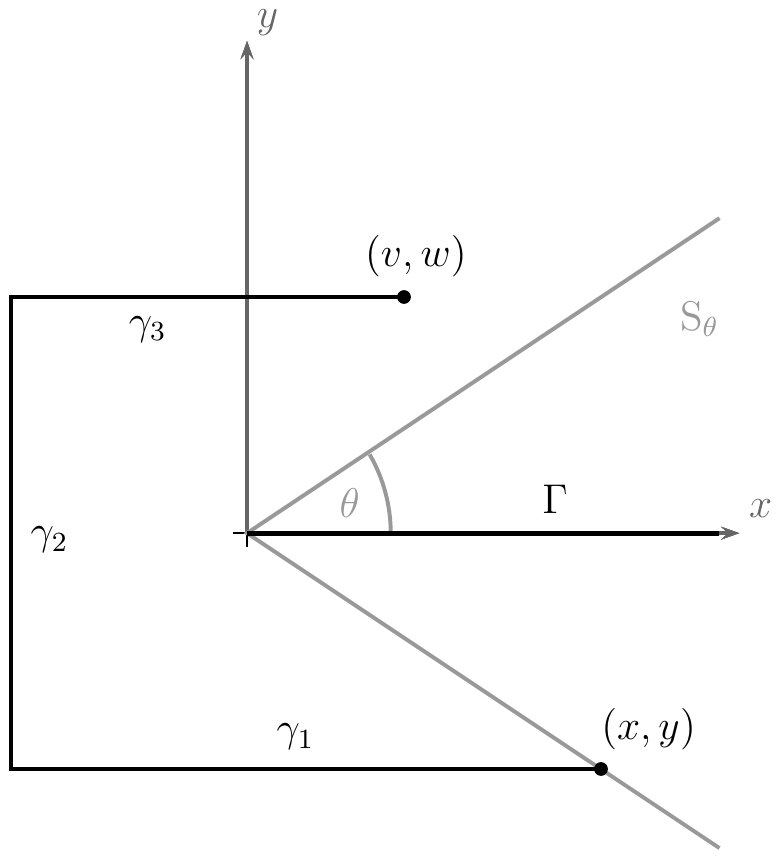}
  \caption{A path connecting $(v , w)$ and $(x , y)$ that is 'short' with respect to the quasihyperbolic distance.}
  \label{Fig: qhdistance}
    \end{figure}
The path $\gamma$ then connects $(x , y)$ to $(v , w)$ and
\begin{align*}
 \qhdist_{\Xi} ((x , y) , (v , w)) \leq \int_0^1 \frac{\lvert y - x \rvert}{\lvert y \rvert} \; \d t + \int_0^1 \frac{\lvert w - y \rvert}{\lvert y \rvert} \; \d t + \int_0^1 \frac{\lvert v - y \rvert}{w} \; \d t = 4 + \frac{x+v}{w}.
\end{align*}
Notice that $x = w / \tan(\theta)$ and that $v \leq w/ \tan(\theta)$, so that
\begin{align*}
 \qhdist_{\Xi} ((x , y) , (v , w)) \leq 2 \Big(2 + \frac{1}{\tan(\theta)}\Big).
\end{align*}
In the remaining cases $v < 0$ and $w \geq 0$, $v < 0$ and $w < 0$, or $v \geq 0$ and $w < 0$ the quasihyperbolic distance to $\Omega$ will even be smaller. This proves the validity of~\eqref{Eq: Distance in example} and thus, since $\Omega$ is connected and hence~\eqref{Eq: Diameter Condition} is void, that $\Omega$ fulfills Assumption~\ref{Ass: Epsilon-delta assumption}.
\end{example}

\begin{remark}
Notice that the geometric setup in Assumption~\ref{Ass: Locally eps-delta near Gamma} imposes boundary regularity in a neighborhood of $\overline{\Gamma}$, while in the situation described in Example~\ref{Ex: Sector} the portion $D$ of $\partial \Omega$ can be arbitrarily irregular as long as it stays outside of $\S_{\theta}$.
\end{remark}

We conclude this section by giving examples of domains where the boundary portion $D$ fails to remain outside of a sector $\S_{\theta}$ and show that the $\W^{1 , p}_D$-extension property fails for these types of domains. These examples show that interior cusps that lie directly on the interface separating $D$ and $\Gamma$ destroy the $\W^{1 , p}_D$-extension property. The same happens with \enquote{interior cusps at infinity}, that is to say, if $D$ and $\Gamma$ approach each other at infinity at a certain rate.

\begin{example}[Interior boundary cusp in zero]
\label{Ex: Interior boundary cusp in zero}
Let $\alpha \in (1 , \infty)$ and consider
\begin{align*}
  \Omega \coloneqq \IR^2 \setminus \{ (x , y) \in \IR^2 \colon x \geq 0 \text{ and } - x^{\alpha} \leq y \leq 0 \}.
\end{align*}
Define $D$ and $\Gamma$ via
\begin{align*}
 D \coloneqq \{ (x , y) \in \IR^2 : x \geq 0 \text{ and } - x^{\alpha} = y \}
 \qquad \text{and} \qquad
 \Gamma \coloneqq (0, \infty) \times \{ 0 \}.
\end{align*}
To prove that the $\W^{1 , p}_D$-extension property fails, let $1 < p < \infty$ and $0 < r < \infty$. Let $f_r$ be a smooth function, that is supported in
\begin{align*}
 Q_r \coloneqq \{ (x , y) \in \IR^2 : r / 2 \leq x \leq 2 r \text{ and } 0 \leq y \leq r \},
\end{align*}
satisfies $0 \leq f_r \leq 1$, and is identically 1 on
\begin{align*}
 R_r \coloneqq \{ (x , y) \in \IR^2 : 3 r / 4 \leq x \leq 3 r / 2 \text{ and } 0 \leq y \leq r / 2 \}.
\end{align*}
Moreover, let $f_r$ be such that $\| \nabla f_r \|_{\L^{\infty}}\lesssim r^{-1}$. In this case
\begin{align}
\label{Eq: Scaling of f}
 \| f_r \|_{\W^{1 , p} (\Omega)}^p \lesssim (r^2 + r^{2 - p}).
\end{align}
Next, employ the fundamental theorem of calculus and a density argument to conclude that for all $F \in \W^{1 , p} (\IR^d)$ it holds
\begin{align*}
 \int_{3 r / 4}^{3 r / 2} F(x , 0) \; \d x - \int_{3 r / 4}^{3 r / 2} F (x , - x^{\alpha}) \; \d x = \int_{3 r / 4}^{3 r / 2} \int_{- x^{\alpha}}^0 \partial_y F(x , y) \; \d y \; \d x.
\end{align*}
If there exists a bounded extension operator $E : \W^{1 , p}_D (\Omega) \to \W^{1 , p}_D (\IR^2)$, put $F \coloneqq E f_r$ and conclude that the second integral on the left-hand side vanishes since $E f_r \in \W^{1 , p}_D (\IR^2)$. Using further that by construction the trace of $E f_r$ onto the set $(3 r / 4 , 3 r / 2) \times \{ 0 \}$ is identically $1$, one concludes
\begin{align*}
 \frac{3 r}{4} \leq \int_{\frac{3 r}{4}}^{\frac{3 r}{2}} \int_{- x^{\alpha}}^0 \lvert \partial_y E f_r (x , y) \rvert \; \d y \; \d x \lesssim r^{(\alpha + 1) / p^{\prime}} \| E f_r \|_{\W^{1 , p} (\IR^2)}.
\end{align*}
Here, $p'$ denotes the H\"{o}lder-conjugate exponent to $p$. Dividing by $r$ and using that $E$ is bounded delivers together with~\eqref{Eq: Scaling of f} the relation
\begin{align}
\label{Eq: Contradictory relation}
 1 \lesssim r^{(\alpha + 1) / p^{\prime} - 1} (r^{2 / p} + r^{2 / p - 1}),
\end{align}
which results  for $r \to 0$ in the condition
\begin{align*}
 \frac{\alpha + 1}{p^{\prime}} + \frac{2}{p} - 2 \leq 0 \qquad \Longleftrightarrow \qquad \alpha \leq 1.
\end{align*}
This is a contradiction since $\alpha$ is assumed to be in $(1 , \infty)$. Thus, there cannot be a bounded extension operator $E : \W^{1 , p}_D (\Omega) \to \W^{1 , p}_D (\IR^2)$.
\end{example}

%\begin{example}[Exterior boundary cusp at infinity]
%Let $\alpha \in (0 , \infty)$ and $R > 0$ and let $\Omega \subseteq \IR^2$ be such that
%\begin{align*}
%  \Omega \cap Q(0 , R)^c = \{ (x , y) \in \IR^2 : x > 0 \text{ and } - x^{- \alpha} < y < 0 \} \cap Q(0 , R)^c.
%\end{align*}
%Define $D$ and $\Gamma$ such that
%\begin{align*}
% D \cap Q(0 , R)^c &= \{ (x , y) \in \IR^2 : x \geq 0 \text{ and } - x^{- \alpha} = y \} \cap Q(0 , R)^c \\
% \Gamma \cap Q(0 , R)^c &= \{ (x , 0) \in \IR^2 : x > 0 \} \cap Q(0 , R)^c.
%\end{align*}
%\end{example}

\begin{example}[Interior boundary cusp at infinity]
\label{Ex: Interior boundary cusp at infinity}
Let $\alpha \in (0 , \infty)$ and consider
\begin{align*}
  \Omega \coloneqq \{ (x , y) \in \IR^2 : \text{either } y > 0 \text{ or } x > 0 \text{  and } y < - x^{- \alpha} \}.
\end{align*}
Define $D$ and $\Gamma$ via
\begin{align*}
 D \coloneqq \{ (x , y) \in \IR^2 : x > 0 \text{ and } - x^{- \alpha} = y \}
\qquad \text{and} \qquad
 \Gamma \coloneqq \IR \times \{ 0 \}.
\end{align*}
The proof that in this situation there does not exists a bounded extension operator $E$ from $\W^{1 , p}_D (\Omega)$ to $\W^{1 , p}_D (\IR^2)$ for any $p \in (1 , \infty)$ is similar to Example~\ref{Ex: Interior boundary cusp in zero} and we omit the details.
%This time, using a parameter $0 < \beta < 1$, we replace $Q_r$ and $R_r$ by
%\begin{align*}
% Q_r^{\beta} \coloneqq \{ (x , y) \in \IR^2 : r / 2 \leq x \leq 2 r \text{ and } 0 \leq y \leq r^{\beta} \}
%\end{align*}
%and
%\begin{align*}
% R_r^{\beta} \coloneqq \{ (x , y) \in \IR^2 : 3 r / 4 \leq x \leq 3 r / 2 \text{ and } 0 \leq y \leq r^{\beta} / 2 \}.
%\end{align*}
%Take a smooth function $f_r^{\beta}$ that satisfies $0 \leq f_r^{\beta} \leq 1$, is identically $1$ on $R_r^{\beta}$, supported in $Q_r^{\beta}$, and satisfies $\| \nabla f_r^{\beta} \|_{\L^{\infty}} \lesssim r^{- \beta}$. As in Example~\ref{Ex: Interior boundary cusp in zero} one derives a relation similar to~\eqref{Eq: Contradictory relation}, which reads in the present situation
%\begin{align*}
% 1 \lesssim r^{(1 - \alpha) / p^{\prime} - 1} (r^{(1 + \beta) / p} + r^{(1 + \beta) / p - \beta}).
%\end{align*}
%For $r \to \infty$ this results in the condition
%\begin{align*}
% \frac{1 - \alpha}{p^{\prime}} - 1 + \frac{1 + \beta}{p} \geq 0 \qquad \Leftrightarrow \qquad \alpha \leq \frac{\beta p^{\prime}}{p}.
%\end{align*}
%As $\beta \in (0 , 1)$ was arbitrary, for $\beta \to 0$ this yields the contradiction $\alpha \leq 0$. Thus, there cannot exist a bounded extension operator $E : \W^{1 , p}_D (\Omega) \to \W^{1 , p}_D (\IR^2)$.
\end{example}

%\begin{figure}[!ht]
%\centering
%\begin{minipage}{.55\textwidth}
%  \centering
%  \includegraphics[width=0.9\textwidth]{interior_cusp.pdf}
%  \caption{Situation in Example~\ref{Ex: Interior boundary cusp in zero}.}
%  \label{Fig: Interior cusp}
%\end{minipage}
%\hfill
%\begin{minipage}{.43\textwidth}
%  \centering
%  \includegraphics[width=0.76\textwidth]{interior_cusp_infinity.pdf}
%  \caption{Situation in Example~\ref{Ex: Interior boundary cusp at infinity}.}
%  \label{Fig: Interior cusp at infinity}
%\end{minipage}
%\end{figure}

\section{Whitney decompositions and the quasihyperbolic distance}
\label{Sec: Whitney decompositions}

\noindent In this section, we introduce the Whitney decomposition of an open subset of $\IR^d$ and show how condition~\eqref{Eq: Quasihyperbolic distance condition} relates to properties of Whitney cubes. A cube $Q \subseteq \IR^d$ is always closed and is said to be
\textit{dyadic} if there exists $k \in \IZ$ such that $Q$ coincides
with a cube of the mesh determined by the lattice $2^{-k} \IZ^d$. Two
cubes are said to \textit{touch} if 
%%the
a face of one cube lies in 
%%the
a face of the other cube, and they are said to \emph{intersect} if their intersection is non-empty. The sidelength of a cube is denoted by
$\ell(Q)$. 
%%Following sentence is edited. 
For a number $\alpha > 0$ the  dilation of $Q$ about its center by the
factor $ \alpha$ is denoted by  $\alpha Q$. \par
Let $F \subseteq \IR^d$ be a non-empty closed set. Then, by~\cite[Thm. VI.1]{Stein} there exists a collection of cubes $\{ Q_j \}_{j \in \IN}$ with pairwise disjoint interiors such that
\begin{enumerate}[labelindent=\parindent, leftmargin=*, label=(\roman*), widest=(iii),align=left, ref=\roman*]
 \item $\bigcup_{j \in \IN} Q_j = \IR^d \setminus F$, \label{Enum: Covering property}
 \item $\diam(Q_j) \leq \dist(Q_j , F) \leq 4 \diam(Q_j)$ for all $j \in \IN$,  \label{Enum: Distance property}
 \item the cubes $\{ Q_j \}_{j \in \IN}$ are dyadic, \label{Eq: Cubes are dyadic}
 \item $\frac{1}{4} \diam(Q_j) \leq \diam(Q_k) \leq 4 \diam(Q_j)$ if $Q_j \cap Q_k \neq \emptyset$, \label{Enum: Comparison property}
 \item each cube has at most $12^d$ intersecting cubes. \label{Enum: Bound Intersecting cubes}
\end{enumerate}
The collection $\{ Q_j \}_{j \in \IN}$ are called 
Whitney cubes and will be denoted by $\WW(F)$. In connection with Whitney cubes, the letters \eqref{Enum: Covering property}-\eqref{Enum: Bound Intersecting cubes} refer always to the above properties. We say that a
collection of cubes $Q_1 , \dots , Q_m \in \WW(F)$ is a
\textit{touching chain} if $Q_j$ and $Q_{j + 1}$ are touching cubes
and that it is an \textit{intersecting chain} if $Q_j \cap Q_{j + 1} \neq
\emptyset$ for all $j = 1 , \dots , m - 1$. The \textit{length} of
%%RMB
a chain is the number $m$.

Let us mention that for a cube $Q\in \WW(F)$ and $x\in Q$ we have $\diam(Q) \geq \frac{1}{5} \dist(x, F)$. This follows from
\begin{align*}
	4\diam(Q) \geq \dist(Q, F) \geq \dist(x, F) - \diam(Q),
\end{align*}
and will be used freely in the rest of this article.
  
The following lemma translates~\eqref{Eq: Quasihyperbolic distance condition} to the existence of intersecting chains of uniformly bounded length. Notice that if $(\Xi_m)_{m \in \mathcal{I}}$ denotes the connected components of the set $\Xi = \IR^d \setminus \overline{\Gamma}$, Gehring and Osgood~\cite[Lem.~1]{Gehring_Osgood} proved that for any two points $x , y \in \Xi_m$ there exists a quasihyperbolic geodesic $\gamma_{x , y}$ with endpoints $x$ and $y$ satisfying
\begin{align*}
 \qhdist_{\Xi}(x , y) = \int_{\gamma_{x , y}} \frac{1}{\dist(z , \partial \Xi)} \; \lvert \d z \rvert.
\end{align*}
Trivially, if $\Xi = \IR^d$, then any path connecting $x$ and $y$ is a quasihyperbolic geodesic.

\begin{lemma}
\label{Lem: Hyperbolic distance to chain condition}
Fix $k > 0$. There exists a constant $N = N(d , k) \in \IN$ such that for all $x , y \in \Xi$ with $\qhdist_{\Xi} (x , y) \leq k$ there exists an intersecting chain $Q_1 , \dots , Q_m \in \WW(\overline{\Gamma})$ with $x \in Q_1$ and $y \in Q_m$ and $m \leq N$. \par
Conversely, if for $x , y \in \Xi$ there exists an intersecting chain connecting $x$ and $y$ of length less than $N \in \IN$, then there exists a constant $k = k(N) > 0$ such that $\qhdist_{\Xi}(x , y) \leq k$.
\end{lemma}

\begin{proof}
Notice that $\qhdist_{\Xi} (x , y) < \infty$ implies that $x$ and $y$ lie in the same connected component of $\Xi$. Assume first that
\begin{align}
\label{Eq: x and y close together}
 \abs{x - y} \leq \frac{1}{10 \sqrt{d}} \min\{ \dist(x , \Gamma) , \dist(y , \Gamma) \}.
\end{align}
Let $Q_x , Q_y \in \WW (\overline{\Gamma})$ with
$x \in Q_x$ and $y \in Q_y$, and let
%%RMB By neighboring cubes, I think you mean ``intersecting cubes'' It
%%is probably best to use the terminology you have already defined,
%%rather than a new term. I will let  you make the final judgement. 
 $\widetilde{Q}_x$ denote the region occupied by $Q_x$ and all its
intersecting Whitney cubes and similarly let $\widetilde{Q}_y$ denote its
counterpart for $Q_y$. Then by~\eqref{Enum: Comparison property} 
%%RMB I think that constant in the inequality below should be
%%1/sqrt{d}. I realize this has no impact on the result. 
%%PT No I think the constant is correct. The 4 in the denominator 
%%comes from property (iii) of the Whitney decomposition
\begin{align*}
 \dist(x , \widetilde{Q}_x^{c}) \geq \frac{1}{4 \sqrt{d}} \diam(Q_x) \qquad \text{and} \qquad \dist(y , \widetilde{Q}_y^{c}) \geq \frac{1}{4 \sqrt{d}} \diam(Q_y).
\end{align*}
This combined with~\eqref{Eq: x and y close together} yields
\begin{align*}
 \dist(x , \widetilde{Q}_x^c) \geq \frac{1}{4 \sqrt{d}} \diam(Q_x) \geq \frac{1}{2} \abs{x - y}.
\end{align*}
By symmetry, the same is valid for $y$ instead of $x$. Consequently, $\widetilde{Q}_x$ and $\widetilde{Q}_y$ have a common point and thus, $x$ and $y$ can be connected by an intersecting chain of length at most $4$. \par
Now, let
\begin{align*}
 \abs{x - y} > \frac{1}{10 \sqrt{d}} \min\{ \dist(x , \Gamma) , \dist(y , \Gamma) \}.
\end{align*}
Assume without loss of generality that $\dist(x , \Gamma) \leq \dist(y , \Gamma)$. Fix a quasihyperbolic geodesic $\gamma_{x , y}$ that connects $x$ with $y$ (see the discussion before this proof). Then Herron and Koskela~\cite[Prop.~2.2]{Herron_Koskela} ensures the existence of points 
$y_0 \coloneqq x , y_1 , \dots , y_\ell \in \IR^d \setminus \overline{\Gamma}$ such that $\gamma_{x,y}$ is contained in the closure of $\bigcup_{i = 0}^\ell B_i$, where $B_i \coloneqq \B(y_i , r_i)$ with $r_i \coloneqq \dist(y_i , \Gamma) / (10 \sqrt{d})$, and such that
\begin{align}
\label{Eq: Construction of finite chains}
\ell \leq 20 \sqrt{d} \qhdist_{\Xi} (x , y).
\end{align}
Next, we estimate the number of Whitney cubes that cover each of these balls. Denote the number of Whitney cubes that cover $\overline{B_i}$ by $W_i$. Let $Q \in \WW(\overline{\Gamma})$ be such that $Q \cap \overline{B_i} \neq \emptyset$. Then,
\begin{align*}
\diam(Q) \geq \frac{1}{4} \dist(Q , \Gamma) \geq \frac{1}{4} [\dist(y_i , \Gamma) - r_i - \diam(Q) ],
\end{align*}
so that by definition of $r_i$
\begin{align*}
 \diam(Q) \geq \frac{(10 \sqrt{d} - 1) \dist(y_i , \Gamma)}{50 \sqrt{d}}.
\end{align*}
Moreover,
\begin{align*}
 \diam(Q) &\leq \dist(Q , \Gamma) \leq \dist(B_i \cap Q, \Gamma) \leq \dist(y_i , \Gamma) + r_i = \bigg[ 1 + \frac{1}{10 \sqrt{d}} \bigg] \dist(y_i , \Gamma).
\end{align*}
Consequently,
\begin{align*}
 W_i \bigg[ \frac{(10 \sqrt{d} - 1) \dist(y_i , \Gamma)}{50 d} \bigg]^d &\leq \sum_{\substack{Q \in \WW(\overline{\Gamma}) \\ Q \cap \cl{B_i} \neq \emptyset}} \abs{Q} \leq \Big\lvert B\Big(y_i , \Big[1 + \frac{1}{5 \sqrt{d}}\Big] \dist(y_i , \Gamma) \Big) \Big\rvert,
\end{align*}
what proves that $W_i$ is controlled by a constant depending only on $d$. We conclude by~\eqref{Eq: Construction of finite chains} and by the bound on each $W_i$ that there exists an intersecting chain connecting $x$ and $y$ of length bounded by a constant depending only on $d$ and $k$. \par
For the other direction, let $Q_1 , \dots , Q_m$ be an intersecting chain that connects $x$ with $y$ and with $m \leq N$. Thus, by definition $Q_j \cap Q_{j + 1} \neq \emptyset$. Let $\gamma$ be a path connecting $x$ and $y$ which is constructed by linearly connecting a point in $Q_{j - 1} \cap Q_j$ with a point in $Q_j \cap Q_{j + 1}$. Thus, employing~\eqref{Enum: Distance property} delivers
\begin{align*}
 \qhdist_{\Xi} (x , y) &\leq \sum_{j = 1}^m \int_{\gamma \cap Q_j} \frac{1}{\dist(Q_j , \Gamma)} \; \lvert \d z \rvert \leq \sum_{j = 1}^m \frac{\diam(Q_j)}{\diam(Q_j)} = m. \qedhere
\end{align*}
\end{proof}

%\begin{lemma}
%If $\Omega \subseteq \IR^d$ is subject to  Assumption \ref{Ass: Epsilon-delta assumption}, then with possibly different $\eps$ and $\delta$, Assumption \ref{Ass: Epsilon-delta assumption} holds true for all $x , y \in \Omega \cup \partial \Omega \setminus \overline{\partial \Omega \setminus D}$.
%\end{lemma}

%\begin{proof}
%If $x , y \in \Omega$, there is nothing to prove. Consider the case, where $x \in \partial \Omega \setminus \overline{\partial \Omega \setminus D}$ and $y \in \Omega$ with $\abs{x - y} < \frac{\delta}{2}$. The case where also $y \in \partial \Omega \setminus \overline{\partial \Omega \setminus D}$ follows the same lines of the proof we are presenting. \par
%Let $\widetilde{x} \in \Omega$ with $\abs{x - \widetilde{x}} < \frac{1}{2} \min\{ \delta , \dist(x , \partial \Omega \setminus D)\}$ and let $\widetilde{\gamma}$ be the path connecting $\widetilde{x}$ and $y$ provided by Assumption \ref{Ass: Epsilon-delta assumption}.
%\end{proof}

%%%%%%%%%%%%%%%%%%%%%%%%%%%%%%%%%%%%%%%%%%%%%%%%%%%%%%%%%%%%%%%%%%%%%%%%%%%%%%%%%%%%%%%%%%%%%%%%%%%%%%%%%%%%%%%%%%%%%%%%%%%%%%%%%%%%%%%%%%%%%%%%%%%%%%%%%%%%%%%%%%%%
\section{Cubes and chains}
\label{Sec: Cubes and chains}

\noindent In this section, we describe how to 'reflect' cubes at $\overline{\Gamma}$ if $\Omega$ is subject to Assumption~\ref{Ass: Epsilon-delta assumption} and establish some natural properties of theses 'reflections'. This is an adaption of an argument of Jones presented in~\cite{Jones}. Throughout, assume in Sections~\ref{Sec: Cubes and chains} and~\ref{Sec: The extension operator} that $\Omega$ is an open set
subject 
%%RMB 
to Assumption~\ref{Ass: Epsilon-delta assumption} which satisfies $\overline{\Omega} \neq \IR^d$. (When $\Omega$ is dense in $\IR^d$, Theorem~\ref{Thm: Extension theorem for large radius} follows in a trivial way. The details will be presented separately in the proof of the theorem).
%We will from now on assume that 
%$\mathrm{radius}_{\Gamma}(\Omega) \geq 1$, i.e., that $r_0 = 1$. For general $r_0 > 0$, Theorem~\ref{Thm: Extension theorem for large radius} then follows by scaling. Moreover, without loss of generality
%assume that $\delta \leq 1$. 
Recall that we assume $\diam(\Omega_m) \geq \lambda \delta$, where $(\Omega_m)_m$ are the connected components of $\Omega$ whose boundary hits $\Gamma$. This is in contrast to~\cite{Jones} where Jones assumes without loss of generality (by scaling) that the domain has radius at least $1$ and that $\delta$ is at most $1$. However, this has the disadvantage that homogeneous estimates will only be achievable on small scales even if $\delta=\infty$ and the domain is unbounded. We will comment on this topic later on in Remark~\ref{Rem: Homogeneous estimates}.

\begin{lemma}
\label{Lem: Measure of Neumann boundary}
We have $\abs{\Gamma} = 0$.
\end{lemma}

\begin{proof}
Fix $x_0 \in \Gamma$ and $y \in \Omega$ with $\abs{x_0 - y} < \frac{\delta}{2}$. Let $Q$ be \textit{any} cube in $\IR^d$ centered in $x_0$ with $l(Q) \leq \frac{1}{2} \abs{x_0 - y}$. We will show that $[\IR^d \setminus \Gamma] \cap Q$ has Lebesgue measure comparable to that of $Q$. Let $x \in \Omega$ with $\abs{x - x_0} \leq \frac{1}{8} l(Q)$. Then, we have
\begin{align}
\label{Eq: Gamma has measure zero}
\abs{x - y} \geq \frac{15}{8} l(Q) \qquad \text{and} \qquad \abs{x - y} \leq \frac{17}{16} \abs{x_0 - y}.
\end{align}
Let $\gamma$ be a path connecting $x$ and $y$ subject to Assumption~\ref{Ass: Epsilon-delta assumption} (note that $|x-y|<\delta$ is either void if $\delta=\infty$ or otherwise it follows from the second inequality in~\eqref{Eq: Gamma has measure zero}).
%TODO Hab hier noch ne kurze Erklaerung bzgl. der Abstandsvoraussetzung eingebaut. Kann man auch wegmachen aber man rechnet ja aus dass der Faktor 17/16 > 1 durch 1/2 zu <1 wird was erst Fishy wirken koennte wenn man eh schon mit infty rechnet, also nur um Verwirrungen vorzubeugen...
By virtue of~\eqref{Eq: Gamma has measure zero}, the intermediate value theorem implies that there exists $z \in \gamma$ with $\abs{x - z} = \frac{1}{8} l(Q)$. This point lies in $\frac{1}{2} Q$ by construction. Moreover,~\eqref{Eq: Carrot condition} together with $\lvert y - z \rvert \geq \lvert x - y \rvert - \lvert x - z \rvert$ implies
\begin{align*}
\dist(z , \Gamma) \geq \frac{\eps l(Q)}{8} \frac{\abs{x - y} - \abs{x - z}}{\abs{x - y}} \ge \frac{\eps l(Q)}{8} ( 1 - \frac{l(Q)}{8|x-y|}) \geq \frac{7 \eps}{60} l(Q).
\end{align*}
Thus, $\limsup_{l(Q) \to 0} \frac{\abs{[\IR^d \setminus \Gamma] \cap Q}}{\abs{Q}} > 0$,
where the $\limsup$ is taken over all cubes centered 
%%RMB replaced in by at
at $x_0$. Since $\chi_{\IR^d \setminus \Gamma}(x_0) = 0$ and $\chi_{\IR^d \setminus \Gamma} \in
\L^1_{\loc}(\IR^d)$ Lebesgue's differentiation theorem %%RMB replaced reveals by implies
implies $\abs{\Gamma} = 0$.
\end{proof}

To proceed, we define two families of cubes. The family of interior cubes is given by
\begin{align*}
 \WW_i \coloneqq \{ Q \in \WW(\overline{\Gamma}) : Q \cap \Omega \neq \emptyset \}.
\end{align*}
These interior cubes will be the reflections of exterior cubes $\WW_e$. To define $\WW_e$ 
we use numbers $A > 0$ and $B > 2$
whose values are to be fixed during this section and define
\begin{align*}
 \WW_e \coloneqq \{Q \in \WW(\overline{\Omega}) : \diam(Q) \leq A \delta \text{ and } \dist(Q , \Gamma) < B \dist(Q , \partial \Omega \setminus \Gamma)\}.
\end{align*}

\begin{remark}
\label{Rem: Neighborhood of Neumann boundary}
First, the collection $\WW_e$ is empty if and only if $D=\bd \Omega$. Indeed, if $D=\bd \Omega$ then the second condition in the definition $\WW_e$ can never be fulfilled. To the contrary, if $\Gamma$ is non-empty, then, using the relative openness of $\Gamma$, one can fix a ball centered in $\Gamma$ that does not intersect $D$, and small cubes inside this ball will satisfy both conditions.
Second, if $D \neq \partial \Omega$, then for a cube $Q \in \WW_e$ we have
\begin{align*}
 \dist(Q , \Omega) = \min\{\dist(Q , \Gamma) , \dist(Q , \partial \Omega \setminus \Gamma)\} \geq B^{-1} \dist(Q , \Gamma)
\end{align*}
what implies that for all $Q \in \WW_e$ it holds
\begin{align}
\label{Eq: Distance Gamma and Omega}
 \dist(Q , \Omega) \leq \dist(Q , \Gamma) \leq B \dist(Q , \Omega).
\end{align}
Thus, the diameter of $Q$ is comparable to its distance to $\Gamma$.
\end{remark}

For the rest of this section, we assume that $\Gamma \neq \emptyset$. Before we present how to 'reflect' cubes, we prove a technical lemma that, given an exterior cube $Q \in \WW_e$, allows us to find a connected component of $\Omega$ whose boundary intersects $\Gamma$ and which is not too far away from $Q$.

\begin{lemma}
\label{Lem: Point in Neumann component}
Let $Q \in \WW_e$. Then there exists a connected component $\Omega_m$ of $\Omega$ with $\Gamma \cap \partial \Omega_m \neq \emptyset$ and $x \in \Omega_m$ with
\begin{align*}
 \dist(x , Q) \leq 5 B \diam(Q).
\end{align*}
\end{lemma}

\begin{proof}
By~\eqref{Enum: Distance property} and Remark~\ref{Rem: Neighborhood of Neumann boundary}, there exists $x^{\prime} \in \overline{\Gamma}$ such that $\dist(x^{\prime} , Q) \leq 4 B
\diam(Q)$. Since $x^{\prime} \in \overline{\Gamma}$ there is $x^{\prime \prime} \in \Gamma$ with $\dist(x^{\prime \prime} , Q) \leq \frac{9}{2} B \diam(Q)$. \par
%%RMB Rewrote the following sentence.
Denote the at most countable family of connected components of $\Omega$ whose boundary has
a non-empty intersection with $\Gamma$ by
$\{\Omega_m\}_m$ and the connected components whose boundary has an empty intersection
with $\Gamma$ by $\{ \Upsilon_m
\}_m$. 

\par
If there is $\Omega_m$ with $x^{\prime \prime} \in \partial \Omega_m$, then the proof is finished. 
% ALTE Variante
%If not, we establish the existence of a sequence $(x_n)_{n \in \IN}$ and indices $m_n$ with $x_n \in \Omega_{m_n}$ and $x_n \to x^{\prime \prime}$ as $n \to \infty$, whose existence still concludes the proof. To this end, assume to the contrary that there exists $\eta > 0$ such that for all $x \in \bigcup_m \Omega_m$
%\begin{align*}
% \abs{x - x^{\prime \prime}} \geq \eta.
%\end{align*}
%However, because $x^{\prime \prime} \in \partial \Omega$ there exists a sequence $(x_n)_{n \in \IN}$ in $\Omega$ with $x_n \to x^{\prime \prime}$ as $n \to \infty$. For $n$ large there must therefore exist indices $m_n$ with $x_n \in \Upsilon_{m_n}$.
% VARIANTE ROBERT
If not, pick a sequence $(x_n)_n$ in $\Omega$ that converges to $x''\in \Gamma \subseteq \bd \Omega$. If almost all $x_n$ are contained in the union of the $\Omega_m$, this concludes the proof as well. Otherwise, choose a subsequence (again denoted by $x_n$) for which there are indices $m_n$ such that $x_n \in \Upsilon_{m_n}$.
Furthermore, $x^{\prime \prime} \in \overline{\Upsilon}_{m}^c$ for all $m$ since $\Gamma \cap \partial \Upsilon_{m} = \emptyset$. Now, by connecting $x^{\prime \prime}$ and $x_n$ by a straight line, the intermediate value theorem implies the existence of a point $x_n^{\prime} \in \partial \Upsilon_{m_n}$ with
\begin{align*}
 \abs{x^{\prime \prime} - x_n^{\prime}} \leq \abs{x^{\prime \prime} - x_n}.
\end{align*}
Passing to the limit $n \to \infty$ yields $x^{\prime \prime} \in D$ by the closedness of $D$ and thus a contradiction.
\end{proof}

The following lemma assigns to every cube in $\WW_e$ a `reflected' cube in $\WW_i$. For the rest of Sections~\ref{Sec: Cubes and chains} and~\ref{Sec: The extension operator} we will reserve the letter $N$ to denote the constant $N$ appearing in Lemma~\ref{Lem: Hyperbolic distance to chain condition} applied with $k = 2 K$, where $K$ is the number from Assumption~\ref{Ass: Epsilon-delta assumption}. Notice that $N$ solely depends on $d$ and $K$. For the rest of this paper we make the following agreement.

\begin{agreement}
\label{Ag: Cube Sec}
If $X$ and $Y$ are two quantities and if there exists a constant $C$ depending only on $d$, $p$, $K$, $\lambda$, and $\eps$ such that $X \leq C Y$ holds, then we will write $X \lesssim Y$ or $Y \gtrsim X$. If both $\frac{Y}{C} \leq X \leq C Y$ holds, then we will write $X \simeq Y$.
\end{agreement}

%\begin{remark}
%The dependence on the parameter $p$ only occurs in Section~\ref{Sec: The extension operator}.
%\end{remark}

\begin{lemma}
\label{Lem: Reflection of cubes}
There exist constants $C_1 = C_1(N , \eps) > 0$ and $C_2=C_2(\lambda) > 0$ such that if $A B \leq C_1$ and $B\geq C_2$, then for every $Q \in \WW_e$ there exists a cube $R \in \WW_i$ satisfying
\begin{align}
\label{Eq: Comparison of cube and reflected cube}
\diam(Q) \leq \diam(R) \lesssim (1 + B + (A B)^{-1}) \diam(Q)
\end{align}
and
\begin{align}
\label{Eq: Estimate distance of reflected cube}
 \dist(R , Q) \lesssim (1 + B + (A B)^{-1}) \diam(Q).
\end{align}
\end{lemma}

\begin{proof}
Fix $Q \in \WW_e$ and recall that $\diam (Q) \leq A \delta$ by definition of $\WW_e$ and $B>2$. By Lemma~\ref{Lem: Point in Neumann component}
there exists a connected component $\Omega_m$ of $\Omega$ with
$\Gamma \cap \partial \Omega_m \neq \emptyset$ and $x \in
\Omega_m$ with $\dist(x , Q) \leq 5 B \diam(Q)$.
We introduce the additional lower bound $B\geq 3/\lambda$, which is only needed in the case $\delta<\infty$ but we choose $B$ always that large for good measure.

So, in the case $\delta<\infty$, since $(A B)^{-1} \diam (Q) \leq \delta B^{-1} < \min(\delta,\lambda \delta/2) \leq \min(\delta, \,\diam(\Omega_m)/2)$ according to~\eqref{Eq: Diameter Condition}, we find $y \in \Omega_m$ satisfying
\begin{align}
\label{Eq: Distance from x ot y}
 \lvert x - y \rvert = (A B)^{-1} \diam (Q) \qquad \text{and} \qquad \lvert x - y \rvert < \delta.
\end{align}
If $\delta=\infty$, then $\Omega$ is unbounded, so we again find $y\in \Omega$ satisfying the first condition whereas the second  becomes void for Assumption~\ref{Ass: Epsilon-delta assumption}.

\indent Hence, let $\gamma$ be a path provided by Assumption~\ref{Ass: Epsilon-delta assumption} connecting $x$ and $y$, and let $z \in \gamma$ with $\abs{x - z} = \frac{1}{2} \abs{x - y}$. Estimate by virtue of~\eqref{Eq: Carrot condition} and~\eqref{Eq: Distance from x ot y}
\begin{align}
\label{Eq: Auxiliary distance from Neumann boundary}
\dist(z , \Gamma) \geq \frac{\eps}{2} \abs{y - z} \geq \frac{\eps}{2} \big(\abs{x - y} - \abs{x - z}\big) = \frac{\eps}{4} (A B)^{-1} \diam(Q).
\end{align}
By Assumption~\ref{Ass: Epsilon-delta assumption} we have $\qhdist_{\Xi}(z , \Omega) \leq K$, hence there exists $z^{\prime} \in \Omega$ with $\qhdist_{\Xi} (z , z^{\prime}) \leq 2 K$. Thus, by Lemma~\ref{Lem: Hyperbolic distance to chain condition} there exists an intersecting chain $Q_1 , \dots , Q_m \in \WW(\overline{\Gamma})$ with $Q_m \cap \Omega \neq \emptyset$, $z \in Q_1$, and $m \leq N$. Choose the reflected cube as $R \coloneqq Q_m$. Using~\eqref{Enum: Distance property} and~\eqref{Enum: Comparison property} one gets
\begin{align*}
4 \diam(R) \geq \dist(R , \Gamma) \geq \dist(z , \Gamma) - \sum_{j = 1}^m \diam(Q_j) \geq \dist(z , \Gamma) - \sum_{j = 1}^m 4^{m - j} \diam(R).
\end{align*}
Thus, by~\eqref{Eq: Auxiliary distance from Neumann boundary} and $m \leq N$
\begin{align*}
 \frac{11 + 4^N}{3} \diam(R) \geq \frac{\eps}{4} (A B)^{-1} \diam(Q).
\end{align*}
Consequently, there exists $C = C(N , \eps) > 0$ such that $A B \leq C$ implies $\diam(Q) \leq \diam(R)$. \par
In order to control $\diam(R)$ by $\diam(Q)$, employ~\eqref{Enum: Distance property},~\eqref{Enum: Comparison property}, and the triangle inequality to deduce
\begin{align*}
 4^{1 - m} \diam(R) \leq \diam(Q_1) \leq \dist(z , \Gamma) \leq \dist(z , Q) + \diam(Q) + \dist(Q , \Gamma).
\end{align*}
The right-hand side is estimated by the triangle inequality, followed by~\eqref{Eq: Distance Gamma and Omega} and~\eqref{Enum: Distance property}, the choice $\lvert x - z \rvert = \frac{1}{2} \lvert x - y \rvert$ combined with~\eqref{Eq: Distance from x ot y}, and $\dist(x , Q) \leq 5 B \diam(Q)$, yielding
\begin{align*}
 \dist(z , Q) + \diam(Q) + \dist(Q , \Gamma) &\leq \abs{z - x} + \dist(x , Q) + \diam(Q) + B \dist(Q , \Omega) \\
 &\leq ( (2 A B)^{-1} + 1 + 9 B ) \diam(Q).
\end{align*}
Taking into account that $\dist(z , R) \leq \diam(R) (4^{m} - 1) / 3$ (estimate the sizes of the cubes in the connecting chain using a geometric sum), the distance from $R$ to $Q$ is estimated similarly, yielding
\begin{align*}
 \dist(R , Q) &\leq \diam(R) + \diam(Q) + \abs{x - z} + \dist(z , R) + \dist(x , Q) \\
 &\leq ( 1 + (2 A B)^{-1} + 5 B ) \diam(Q) + \frac{4^{m} + 2}{3} \diam(R).
\end{align*}
Together with the previous estimate, this concludes the proof.
\end{proof}

For the rest of this article, we fix the notation that if $Q \in \WW_e$ and $R \in \WW_i$ is the cube constructed in Lemma~\ref{Lem: Reflection of cubes}, then $R$ is denoted by $R = Q^*$ and $Q^*$ is called the \textit{reflected} cube of $Q$. The next lemma gives a bound on the distance of reflected cubes of two intersecting cubes. Its proof is a direct consequence of Lemma~\ref{Lem: Reflection of cubes} and~\eqref{Enum: Comparison property}, and is thus omitted.

\begin{lemma}
\label{Lem: Distance reflected cubes}
If $Q_1 , Q_2 \in \WW_e$ with $Q_1 \cap Q_2 \neq \emptyset$, then
\begin{align*}
 \dist(Q_1^* , Q_2^*) \lesssim (1 + B + (A B)^{-1}) \diam(Q_1).
\end{align*}
\end{lemma}

In the proof of the boundedness of the extension operator, one needs to connect Whitney cubes by appropriate touching chains. The following lemma presents a basic principle of how to build a chain out of a path $\gamma$ and how the quantities $\mathrm{length}(\gamma)$ and $\dist(\gamma , \Gamma)$ translate into the length of the chain and the distance of the cubes of the chain to $\Gamma$.

%%RMB Do we need to update this lemma after introducing
%%quasi-hyperbolic distance? No
\begin{lemma}
\label{Lem: Basic chain lemma}
Let $R_1 , R_2 \in \WW(\overline{\Gamma})$ with $R_1 \neq R_2$ and let $x \in R_1$, $y \in R_2$, and $\gamma$ be a rectifiable path in $\IR^d \setminus \overline{\Gamma}$ connecting $x$ and $y$. Assume that there exist constants $C_1 , C_2 > 0$ such that $\mathrm{length}(\gamma) \leq C_1 \diam(R_1)$ and $\dist(z , \Gamma) \geq C_2 \diam(R_1)$ for all $z \in \gamma$, then there exists a touching chain of cubes $R_1 = S_1 , \dots , S_m = R_2$ in $\WW(\overline{\Gamma})$, where $m$ is bounded by a number depending only on $d$, $C_1$, and $C_2$. Moreover,
\begin{align*}
 \frac{C_2}{5} \diam(R_1) \leq \diam(S_i) \leq (5 + C_1) \diam(R_1) \qquad (i = 1 , \dots , m).
\end{align*}
\end{lemma}

\begin{proof}
Let $\mathcal{S}$ be the \textit{finite} set of cubes in $\WW(\overline{\Gamma})$ intersecting $\gamma$. For $S \in \mathcal{S}$ one finds by~\eqref{Enum: Distance property} and by assumption that $\diam(S) \geq \frac{C_2}{5} \diam(R_1)$. Fix $z \in S \cap \gamma$, then
\begin{align*}
 \dist(z , \Gamma) \leq \dist(x , \Gamma) + \abs{x - z} \leq 5 \diam(R_1) + \mathrm{length}(\gamma) \leq (5 + C_1) \diam(R_1),
\end{align*}
so that $\diam(S) \leq (5 + C_1) \diam(R_1)$ by~\eqref{Enum: Distance property}. This, together with $\mathrm{length} (\gamma) \leq C_1 \diam(R_1)$ implies that $S \subseteq \B(x , (5 + 2 C_1) \diam(R_1))$. Because all elements of $\mathcal{S}$ are mutually disjoint one finds
\begin{align*}
 \sharp (\mathcal{S}) \leq \frac{\abs{\B(x , (5 + 2 C_1) \diam(R_1))}}{(\frac{C_2}{5 \sqrt{d}} \diam(R_1))^d} = \omega_d \bigg( \frac{5 \sqrt{d} (5 + 2 C_1)}{C_2} \bigg)^d,
\end{align*}
where $\sharp(\mathcal{S})$ denotes the cardinality of $\mathcal{S}$ and $\omega_d \coloneqq \abs{\B(0 , 1)}$. By~\eqref{Eq: Cubes are dyadic} the elements of $\mathcal{S}$ are dyadic and thus one finds a subset of $\mathcal{S}$ which is a touching chain starting at $R_1$ and ending at $R_2$.
\end{proof}

\begin{lemma}
\label{Lem: Building chains}
There exist constants $C_1 , C_2 > 0$ depending only on $\eps$, $d$, $\lambda$, and $K$ such that
if $A \leq C_1$ and $B \geq C_2$ and if $Q_j , Q_k
\in \WW_e$ with $Q_j \cap Q_k \neq \emptyset$, then there exists a touching
chain $F_{j , k} = \{ Q_j^* = S_1 , \dots , S_m = Q_k^*\}$ of cubes in
$\WW(\overline{\Gamma})$ connecting $Q_j^*$ and
$Q_k^*$, where $m$ can be bounded uniformly by a constant depending only
on $\eps$, $d$, $K$, $A$, and $B$. Moreover, there exist $K_1 , K_2 > 0$ depending only on $\eps$, $d$, $K$, $A$, and $B$ such that
\begin{align*}
 K_1 \diam(Q_j) \leq \diam(S_i) \leq K_2 \diam(Q_j) \qquad (i = 1 , \dots , m).
\end{align*}
\end{lemma}

\begin{proof}
If $Q_j^* = Q_k^*$ there is nothing to show. Thus, assume $Q_j^* \neq Q_k^*$ and let $x \in Q_j^\ast$ and $y \in Q_k^\ast$. We show in the following that the assumptions of Lemma~\ref{Lem: Basic chain lemma} are satisfied. \par
%TODO Die Diameter Bedingung scheint mir irrelevant zu sein, um die Argument-Dichte zu erhalten habe ich dafuer einen Verweis auf die Groesse benachbarter Whitney-Wuerfel eingebaut
%Fix $x \in Q_j^* \cap \Omega$ and $y \in Q_k^* \cap \Omega$. Since $\diam(Q_j) \leq A \delta$ one obtains by Lemmas~\ref{Lem: Reflection of cubes} and~\ref{Lem: Distance reflected cubes}
Using Lemmas~\ref{Lem: Reflection of cubes} and~\ref{Lem: Distance reflected cubes} in conjunction with~\eqref{Enum: Comparison property} gives
\begin{align}
\label{Eq: Ensure that x and y are close}
 \abs{x - y} \leq \dist(Q_j^* , Q_k^*) + \diam(Q_j^*) + \diam(Q_k^*) \lesssim (1 + B + (A B)^{-1}) \diam(Q_j).
\end{align}
%TODO hier ist nur der Fall delta endlich relevant
If $\delta$ is finite we get, since $Q_j\in \WW_e$, that
\begin{align*}
	|x-y| \lesssim (A+AB+B^{-1}) \delta,
\end{align*}
so we obtain $|x-y| < \delta$ when we first choose $B$ large enough and afterwards $A$ sufficiently small.
%Estimating $\diam(Q_j)$ by $A \delta$ (this is possible since $Q_j \in \WW_e$) shows that choosing first $B$ large enough and then $A$ small enough ensures $\lvert x - y \rvert < \delta$.
Let $\gamma$ be a path connecting $x$ and $y$ according to Assumption~\ref{Ass: Epsilon-delta assumption}. By~\eqref{Eq: Length condition},~\eqref{Eq: Ensure that x and y are close}, and~\eqref{Eq: Comparison of cube and reflected cube} one finds
\begin{align*}
 \mathrm{length} (\gamma) \lesssim (1 + B + (A B)^{-1}) \diam(Q_j^*).
\end{align*}
To estimate the distance between each $z \in \gamma$ and $\Gamma$, notice that if $\abs{x - z} \leq \frac{1}{2} \diam(Q_j^*)$, then $\dist(z , \Gamma) \geq \frac{1}{2} \diam(Q_j^*)$. Analogously, but by employing additionally~\eqref{Eq: Comparison of cube and reflected cube} twice and~\eqref{Enum: Comparison property}, if $\abs{y - z} \leq \frac{1}{2} \diam(Q_k^*)$, then 
\begin{align}
\label{Eq: Comparing reflected cubes that are close together}
 \dist(z , \Gamma) \geq \frac{1}{2} \diam(Q_k^*) \geq \frac{1}{2} \diam(Q_k) \geq \frac{1}{8} \diam(Q_j) \gtrsim \big(1 + B + (AB)^{-1} \big)^{-1} \diam(Q_j^*).
\end{align}
In the remaining case, one estimates by~\eqref{Eq: Carrot condition}, the calculation performed in~\eqref{Eq: Comparing reflected cubes that are close together}, and~\eqref{Eq: Ensure that x and y are close} that
\begin{align*}
 \dist(z , \Gamma) &\gtrsim \frac{\diam(Q_j^*)^2}{(1 + B + (AB)^{-1}) \abs{x - y}} \gtrsim \frac{\diam(Q_j^*)}{(1 + B + (A B)^{-1})^2}. \qedhere
\end{align*}
\end{proof}

The following lemma provides the existence of chains that `escape $\Omega$' for reflections of cubes $Q \in \WW(\overline{\Omega})$ that are close to a relatively open portion of $D$. These chains will be important to obtain a Poincar\'e inequality with a quantitative control of the constants.

\begin{lemma}
\label{Lem: Building Poincare chains}
There exist constants $C_1 , C_2 > 0$ depending only on $\eps$, $d$, $\lambda$, and $K$ such that
if $A \leq C_1$ and $B \geq C_2$ and if $Q \in \WW(\overline{\Omega}) \setminus \WW_e$ satisfies $\diam(Q) \leq A \delta$ and has a non-empty intersection with a cube in $\WW_e$, then for each intersecting cube $Q_j \in \WW_e$ of $Q$, there exists a touching chain $F_{P , j} = \{Q_j^* = S_1 , \dots , S_m \}$ of cubes in $\WW(\overline{\Gamma})$, where $m$ is bounded by a constant depending only on $\eps$, $d$, $K$, $A$, and $B$ and $S_m \cap Q_j$ is a dyadic cube that satisfies
\begin{align*}
 \abs{S_m \cap Q_j} \gtrsim \diam(Q_j)^d.
\end{align*}
Furthermore, all $S_i \in F_{P , j}$ satisfy
\begin{align*}
 K_1 \diam(Q_j) \leq \diam(S_i) \leq K_2 \diam(Q_j) \qquad (i = 1 , \dots , m).
\end{align*}
The constants $K_1 , K_2 > 0$ depend only on $\eps$, $d$, $K$, $A$, and $B$.
\end{lemma}

\begin{proof}
Let $Q_j \in \WW_e$ be an intersecting cube of $Q$. Then, using properties of the Whitney cubes and $Q \notin \WW_e$, one estimates
\begin{align}
\label{Eq: Distance realized at D}
 B \dist(Q_j , \partial \Omega \setminus \Gamma) \leq 6 B \dist(Q , \partial \Omega \setminus \Gamma) \leq 6 \dist(Q , \Gamma) \leq 36 \dist(Q_j , \Gamma).
\end{align}
Let $B \geq 720$, then~\eqref{Eq: Distance realized at D} and Remark~\ref{Rem: Neighborhood of Neumann boundary} imply $\dist(Q_j , \partial \Omega \setminus \Gamma) = \dist(Q_j , \Omega)$. Hence, using~\eqref{Eq: Distance realized at D} again and~\eqref{Enum: Comparison property}, one finds that $\dist(Q_j , \Gamma) \geq \frac{B}{36} \dist(Q_j , \Omega) \geq \frac{B}{36} \diam(Q_j)$. Let $x_0 \in D$ be such that $\dist(x_0 , Q_j) = \dist(Q_j , \Omega)\leq 4\diam(Q_j)$. The properties collected above then imply
\begin{align*}
 \dist(x_0 , \Gamma) \geq \dist(Q_j , \Gamma) - \dist(x_0 , Q_j) - \diam(Q_j) \geq (36^{-1} B - 5) \diam(Q_j),
\end{align*}
and if $y$ is any point from $\B(x_0 , 5 \diam(Q_j))$ then the previous estimate delivers
\begin{align*}
 \dist(y , \Gamma) \geq \dist(x_0 , \Gamma) - 5 \diam(Q_j) \geq (36^{-1} B - 10) \diam(Q_j) \geq 10 \diam(Q_j).
\end{align*}
%TODO habe das nochmal gegen 10 \diam(Q_j) nach unten abgeschätzt, hast du dazu nicht B \geq 720 gewaehlt?
Fix $y \in \B(x_0 , 5 \diam(Q_j)) \cap \Omega$.
Notice that the midpoint $z$ of $Q_j$ is contained in $\B(x_0 , 5 \diam(Q_j))$. Thus, each point on the line segment $\gamma_1$ connecting $y$ to $z$ has at least a distance which is larger than $10 \diam(Q_j)$ to $\Gamma$. \par
For $x \in Q_j^* \cap \Omega$ Lemma~\ref{Lem: Reflection of cubes} together with $\{ y \} \cup Q_j \subseteq \B(x_0 , 6 \diam(Q_j))$ implies
%TODO hab beim Radius aus der 5 eine 6 gemacht, ich glaube das braucht man weil die Wuerfel abgeschlossen und der Ball offen ist, schaden tut es jedenfalls nicht.
\begin{align*}
 \abs{x - y} \leq \dist(y , Q_j) + \dist(Q_j , Q_j^*) + \diam(Q_j) + \diam(Q_j^*) \lesssim (1 + B + (A B)^{-1}) \diam(Q_j).
\end{align*}
%TODO hier wieder constraints an A und B falls delta endlich.
If $\delta$ is finite, we can ensure $|x-y|<\delta$ using exactly the same argument as in the proof of Lemma~\ref{Lem: Building chains} and otherwise this condition is again meaningless.
%Recall that by definition of $\WW_e$ it holds $\diam (Q_j) \leq A \delta$. Choosing $B$ large enough and $A$ small enough delivers  then that $\lvert x - y \rvert < \delta$.
Let $\gamma_2$ be the path connecting $x$ and $y$ subject to Assumption~\ref{Ass: Epsilon-delta assumption} and let $Q \in \WW (\overline{\Gamma})$ with $y \in Q$. Since $\dist(Q , \Gamma) \geq C \diam(Q_j^*)$ for some $C > 0$ depending only on $\eps$, $K$, $d$, $A$, and $B$, one concludes as in the proof of Lemma~\ref{Lem: Building chains} that the path $\gamma_2$, and hence, by the consideration above, also the path $\gamma = \gamma_1 + \gamma_2$ which connects $x\in Q_j^\ast$ with $z\in Q_j$ satisfies the assumptions of Lemma~\ref{Lem: Basic chain lemma}, where $Q_j^\ast$ fulfills the role of $R_1$ and $R_2$ is some cube in $\WW(\overline{\Gamma})$ that contains $z$. Note that the constants appearing in Lemma~\ref{Lem: Basic chain lemma} depend only on $\eps$, $K$, $d$, $A$, and $B$. \par
As in the statement of the lemma we write $S_m$ for $R_2$ and distinguish cases for the relation between $S_m$ and $Q_j$. Since $Q_j \cap S_m \neq \emptyset$ and since Whitney cubes are dyadic, it either holds $S_m \subseteq Q_j$ or $Q_j \subseteq S_m$. If $Q_j \subseteq S_m$ the proof is finished. If $S_m \subseteq Q_j$, then
\begin{align*}
 4 \diam(S_m) \geq \dist(S_m , \Gamma) \geq \dist(Q_j , \Gamma) \geq 36^{-1} B \diam(Q_j),
\end{align*}
so that $\abs{S_m \cap Q_j} \gtrsim \diam(Q_j)^d$.
\end{proof}

The next lemma shows that for a fixed cube $R \in \WW_i$, there are only finitely many cubes in $\WW_e$, whose reflected cube is $R$.

\begin{lemma}
\label{Lem: Number of preimages to reflected cube}
There is a constant $C \in \IN$ such that for each $R \in \WW_i$ there are at most $C$ cubes $Q \in \WW_e$ such that $Q^* = R$, where $C$ solely depends on $d$, $K$, $A$, $B$, and $\eps$.
\end{lemma}

\begin{proof}
Let $\alpha$ denote the constant from~\eqref{Eq: Estimate
  distance of reflected cube} and let $Q\in \WW_e$ with reflected cube $R$, then it follows with~\eqref{Eq:
  Comparison of cube and reflected cube} that $\dist(R,Q) \leq \alpha \diam(R)$. So, if $x_R$ denotes the center of $R$,
every cube $Q$ with $Q^* = R$ must be contained in $\B(x_R , (\alpha +
\frac{3}{2}) \diam(R))$. Because for those cubes $\diam(Q)$ is
controlled from below by $\diam(R)$ according to~\eqref{Eq:
  Comparison of cube and reflected cube} and because cubes from
$\WW_e$ have disjoint interiors, the lemma follows by a counting argument.
\end{proof}

%%%%%%%%%%%%%%%%%%%%%%%%%%%%%%%%%%%%%%%%%%%%%%%%%%%%%%%%%%%%%%%%%%%%%%%%%%%%%%%%%%%%%%%%%%%%%%%%%%%%%%%%%%%%%%%%%%%%%%%%%%%%%%%%%%%%%%%%%%%%%%%%%%%%%%%%%%%%%%%%%%%%
\section{Construction of the extension operator and exterior estimates}
\label{Sec: The extension operator}

\noindent This section is devoted to the construction of the extension operator from Theorem~\ref{Thm: Extension theorem for large radius}. We also establish estimates for the operator on $\overline{\Omega}^c$. To do so, we start with a preparatory part on (adapted) polynomials, followed by some overlap considerations. We proceed with the construction of the extension operator, in which the adapted polynomials will appear, followed by the exterior estimates, for which we will need the results on overlap.

\begin{agreement}
If not otherwise mentioned, the symbols $k$ and $p$ are supposed to refer to these parameters in Theorem~\ref{Thm: Extension theorem for large radius}. The numbers $A$ and $B$ which were introduced in Section~\ref{Sec: Cubes and chains} will be considered as fixed numbers depending only on $\eps$, $d$, $K$, and $\lambda$ such that all statements in Section~\ref{Sec: Cubes and chains} are valid. From now on we will use the symbols $\lesssim$ and $\gtrsim$ in a more liberal way than described in Agreement~\ref{Ag: Cube Sec}.
\end{agreement}

\subsection*{Polynomial fitting and Poincar\'e type estimates}
We record some results on polynomial approximation and Poincar\'e type estimates. Most of them stem from~\cite{Chua} and where used therein for a similar purpose.

We start out with the following generic norm comparison lemma for polynomials of fixed degree, see~\cite[Lem.~2.3]{Chua}.
\begin{lemma}
\label{Lem: Polynomial norm on different cubes}
	Let $Q,R$ be cubes with $R\subseteq Q$ and assume that there exists a constant $\kappa>0$ such that $|R| \geq \kappa |Q|$. Then for each polynomial $P$ of degree at most $m$ one has
	\begin{align*}
		\| P \|_{\L^p(Q)} \lesssim \| P \|_{\L^p(R)},
	\end{align*}
	where the implicit constant does only depend on $d$, $\kappa$, $p$, and $m$. In particular, if $S$ is another cube with $S\subseteq Q$ and $|S| \geq \kappa |Q|$, then the $\L^p$ norms over $R$ and $S$ are equivalent norms on $\mathcal{P}_{m}$ and the implicit constants do only depend on $d$, $\kappa$, $p$, and $m$.
\end{lemma}
The following lemma provides \enquote{adapted} polynomials together with corresponding Poincar\'e type estimates. A proof can be found in~\cite[Thm.~4.5, Thm.~4.7 \& Rem.~4.8]{Chua}, note that the proof of the remark still works when replacing Whitney cubes by cubes of the same size.
\begin{lemma}
\label{Lem: Polynomial Poincare}
	Let $Q$ be a cube, $R$ a touching cube of $Q$ of the same size, and $k\geq 0$ an integer. Then there exists a projection $P: \L^1(Q) \to \mathcal{P}_{k-1}$ that satisfies the estimate
	\begin{align}
	\label{Eq: Polynomial Proj Estimate}
		\| \partial^\alpha Pf \|_{\L^p(Q)} \lesssim \| \partial^\alpha f \|_{\L^p(Q)}
	\end{align}
	for $f\in \C_D^\infty(\IR^d) \cap \W^{k,p}(Q)$, $|\alpha| \leq k$ and $1\leq p \leq \infty$. Moreover, the Poincar\'e type estimate
	\begin{align}
	\label{Eq: Polynomial Proj Poincare}
		\| \partial^\alpha (f-Pf) \|_{\L^p(Q\cup R)} \lesssim \diam(Q)^{\ell-|\alpha|} \| \nabla^\ell f \|_{\L^p(Q\cup R)},
	\end{align}
	holds for $f\in \C_D^\infty(\IR^d) \cap \W^{k,p}(Q\cup R)$, $0\leq \ell \leq k$, $|\alpha| \leq \ell$, and $1\leq p \leq \infty$.
	The implicit constants depend only on $d$, $k$, and $p$. Of course, the case $Q=R$ is also permitted.
\end{lemma}
\begin{remark}
\label{Rem: Adapted Polynomial}
\begin{enumerate}
	\item The polynomial $Pf$ will be denoted by $(f)_Q$. The case $|\alpha|=k$ in~\eqref{Eq: Polynomial Proj Estimate} was not stated in~\cite{Chua} but follows since the degree of $Pf$ is at most $k-1$.
	\item That the projection is always meaningfully defined on $\L^1(Q)$ becomes evident from $(4.2)$ in~\cite{Chua}.
	\item In the case $\alpha=0$ we can drop the intersection with $\C_D^\infty(\IR^d)$ in both estimates in Lemma~\ref{Lem: Polynomial Poincare}. This follows from a direct computation using the representation formula for the projection given in~\cite{Chua}.
\end{enumerate}
	
\end{remark}
Combining these results gives a Poincar\'e type estimate where the polynomial is only adapted to a subcube of the domain of integration.
\begin{corollary}
\label{Cor: Poincare on subcubes}
	Let $Q$ and $R$ be cubes with $R\subseteq Q$ such that there is $\kappa>0$ with $|R| \geq \kappa |Q|$. Then with $f\in \C_D^\infty(\IR^d)\cap\W^{k,p}(Q)$, $0\leq \ell \leq k$, $|\alpha| \leq \ell$, and $1\leq p\leq \infty$ we obtain
	\begin{align*}
		\| \partial^\alpha (f-(f)_R) \|_{\L^p(Q)} \lesssim \diam(Q)^{\ell-|\alpha|} \| \nabla^\ell f \|_{\L^p(Q)},
	\end{align*}
	where the implicit constant does only depend on $d$, $k$, $p$, and $\kappa$. We may also replace $Q$ by the union of two touching cubes of the same size where one of them contains $R$ as a subcube.
\end{corollary}

\begin{proof}
	\textbf{Step 1}. We start with the case that $Q$ is a single cube. Using Lemma~\ref{Lem: Polynomial Poincare} and Lemma~\ref{Lem: Polynomial norm on different cubes} we get from the fact that $\partial^\alpha ((f)_Q-(f)_R)$ is a polynomial of degree at most~$k-1$ the estimate
	\begin{align*}
		\| \partial^\alpha (f-(f)_R) \|_{\L^p(Q)} &\leq \| \partial^\alpha (f-(f)_Q) \|_{\L^p(Q)} + \| \partial^\alpha ((f)_Q-(f)_R) \|_{\L^p(Q)} \\
		&\lesssim \diam(Q)^{\ell-|\alpha|} \| \nabla^\ell f \|_{\L^p(Q)} + \| \partial^\alpha ((f)_Q-(f)_R) \|_{\L^p(R)}.
	\end{align*}
	The first term is already fine so we focus on the second one. Using that $R\subseteq Q$ and Lemma~\ref{Lem: Polynomial Poincare} twice, we estimate further
	\begin{align*}
		\| \partial^\alpha ((f)_Q-(f)_R) \|_{\L^p(R)} &\leq \| \partial^\alpha (f-(f)_Q) \|_{\L^p(R)} + \| \partial^\alpha (f-(f)_R) \|_{\L^p(R)} \\
		&\lesssim \| \partial^\alpha (f-(f)_Q) \|_{\L^p(Q)} + \diam(R)^{\ell-|\alpha|} \| \nabla^\ell f \|_{\L^p(R)} \\
		&\lesssim \diam(Q)^{\ell-|\alpha|} \| \nabla^\ell f \|_{\L^p(Q)}.
	\end{align*}
	\textbf{Step 2}. Now, assume that $Q=Q_1 \cup Q_2$ is the union of two touching cubes with $R\subseteq Q_1$ and $|R| \geq \kappa |Q_1|$. We reduce this case to the already shown case. Start with the triangle inequality to get
	\begin{align*}
		\| \partial^\alpha (f-(f)_R) \|_{\L^p(Q)} \leq \| \partial^\alpha (f-(f)_R) \|_{\L^p(Q_1)} + \| \partial^\alpha (f-(f)_R) \|_{\L^p(Q_2)}.
	\end{align*}
	The first term is fine by Step~1 and for the second we continue with
	\begin{align*}
		\| \partial^\alpha (f-(f)_R) \|_{\L^p(Q_2)} \leq \| \partial^\alpha (f-(f)_{Q_2}) \|_{\L^p(Q_2)} + \| \partial^\alpha ((f)_{Q_2}-(f)_R) \|_{\L^p(Q_2)}.
	\end{align*}
	Again, the first term is good and for the other one we exploit $Q_1 \subseteq 3 Q_2$ to derive with Lemma~\ref{Lem: Polynomial norm on different cubes}
	\begin{align*}
		 \| \partial^\alpha ((f)_{Q_2}-(f)_R) \|_{\L^p(Q_2)} &\lesssim \| \partial^\alpha ((f)_{Q_2}-(f)_R) \|_{\L^p(Q_1)} \\
		 &\leq \| \partial^\alpha (f-(f)_{Q_2}) \|_{\L^p(Q)} + \| \partial^\alpha (f-(f)_R) \|_{\L^p(Q_1)}.
	\end{align*}
	The first term is good by Lemma~\ref{Lem: Polynomial Poincare} and the second by Step~1.
\end{proof}

%\begin{proposition}
%\label{Prop: Poincare}
%Let $\Upsilon \subseteq \IR^d$ be open, bounded, and convex and $S \subseteq \Upsilon$ be measurable with $\lvert S \rvert > 0$. Then for all $1 \leq p \leq \infty$ the inequality
%\begin{align*}
%\|u - (u)_S\|_{\L^p(\Upsilon)} \leq \frac{\omega_d^{1 - \frac{1}{d}} \diam(\Upsilon)^d \lvert \Upsilon \rvert^{\frac{1}{d}}}{\lvert S \rvert} \|\nabla u\|_{\L^p(\Upsilon)} \qquad (u \in \W^{1 , p}(\Upsilon))
%\end{align*}
%holds true, where $\omega_d = \lvert \B(0 , 1) \rvert$.
%\end{proposition}

\subsection*{Some overlap considerations}
Let $Q_j \in \WW_e$ and let $Q \in \WW(\overline{\Omega}) \setminus \WW_e$ be such that it intersects a cube in $\WW_e$ and satisfies $\diam(Q) \leq A \delta$. Define
\begin{align*}
 F(Q_j) \coloneqq \bigcup_{\substack{Q_k \in \WW_e\\ Q_j \cap Q_k \neq \emptyset}} \bigcup_{S \in F_{j , k}} 2 S \qquad \text{and} \qquad F_P(Q) \coloneqq \bigcup_{\substack{Q_k \in \WW_e\\ Q \cap Q_k \neq \emptyset}} \bigcup_{S \in F_{P , k}} 2 S.
\end{align*}
We count how many of these \enquote{extended} chains $F(Q_j)$ and $F_P (Q)$ can intersect a fixed point $x \in \IR^d$. To be concise, we only present the case of $F(Q_j)$. 

By Lemma~\ref{Lem: Building chains}, we know that a chain $F_{j , k}$ has length less than a constant $M$ which does only depend on $d$, $K$, $\lambda$, and $\eps$. If $x \in F(Q_j)$, then there exist $k\in \IN$ and $S \in F_{j , k}$ with $x \in 2 S$. Assume $R \in \WW(\overline{\Gamma})$ is any cube such that also $x\in 2 R$.
By~\eqref{Enum: Distance property} and an elementary geometric consideration one infers for $z \in S$ that
\begin{align*}
 4 \diam(R) \geq \dist(R , \Gamma) \geq \dist(z , \Gamma) - \lvert x - z \rvert - \frac{3}{2} \diam(R).
\end{align*}
Pick some $z$ that satisfies $\lvert x - z \rvert \leq \diam(S) / 2$. Then
\begin{align*}
 4 \diam(R) \geq \dist(S, \Gamma) - \frac{1}{2} \diam(S) - \frac{3}{2} \diam(R) \geq \frac{1}{2} \diam(S) - \frac{3}{2} \diam(R).
\end{align*}
By symmetry (interchange $S$ and $R$) this implies that
\begin{align}
\label{Eq: Size extended chains}
 \frac{1}{11} \diam (S) \leq \diam (R) \leq 11 \diam (S).
\end{align}
%TODO ich fand die urspruengliche Formulierung vom Ende vom Beweis sehr schwer verstaendlich, deshalb habe ich mich mal an einer Umformulierung versucht, schau was du davon haelst, alte Version ist auskommentiert drunter.
Now let $F_{\alpha,\beta}$ be another chain such that $x\in \cup_{S\in F_{\alpha,\beta}} 2S$. This means that there is a cube in $F_{\alpha,\beta}$ that fulfills the role of $R$ above. Since $Q_\alpha^\ast$ and $R$ as well as $Q_j^\ast$ and $S$ are connected by touching chains of Whitney cubes each of length at most $M$, we deduce from~\eqref{Eq: Size extended chains} that $\diam(Q_\alpha^\ast) \approx \diam(Q_j^\ast)$ and conclude $\dist(Q_j^\ast, Q_\alpha^\ast) \lesssim \diam(Q_j^\ast)$. Then the usual counting argument yields a bound on such reflected cubes $Q_\alpha^\ast$. Finally, Lemma~\ref{Lem: Number of preimages to reflected cube} implies that there exists a constant $C > 0$ that depends only on $d$, $K$, $\lambda$, and $\eps$ such that
\begin{align}
\label{Eq: Nasty chain estimate}
 \sum_{Q_j \in \WW_e} \chi_{F(Q_j)} (x) \leq C.
\end{align}
%So, if $R \in F_{\alpha , \beta}$ for some $\alpha , \beta \in \IN$, then $R$ can be connected to $Q_{\alpha}^*$ by a touching chain of length at most $M$. Thus, by~\eqref{Eq: Size extended chains} and~\eqref{Enum: Comparison property} any cube $R^{\prime}$ in the chain $F_{\alpha, \beta}$ satisfies
%\begin{align*}
% 11^{-1} 4^{- 2 M} \leq \frac{\diam(Q_j^*)}{\diam (R^{\prime})} \leq 11 \cdot 4^{2 M} \qquad \text{and} \qquad  \dist(Q_j^* , R^{\prime}) \leq 11 (2 M + 1) 4^{2 M} \diam (Q_j^*).
%\end{align*}
%This combined with Lemma~\ref{Lem: Number of preimages to reflected cube} implies that there exists a constant $C > 0$ that depends only on $d$, $K$, and $\eps$ such that
%\begin{align}
%\label{Eq: Nasty chain estimate}
% \sum_{Q_j \in \WW_e} \chi_{F(Q_j)} (x) \leq C.
%\end{align}

\subsection*{Construction of the extension operator}
Fix an enumeration $(Q_j)_j$ of $\WW_e$ and take a partition of unity $(\varphi_j)_j$ on $\bigcup_{Q_j \in \WW_e} Q_j$ valued in $[0,1]$ and satisfying $\supp(\varphi_j) \subseteq \frac{17}{16} Q_j$ as well as $\| \partial^\alpha \varphi_j \|_{\L^{\infty}} \lesssim \diam(Q_j)^{-|\alpha|}$ for $|\alpha| \leq k$ and with an implicit constant only depending on $k$.

Let $f$ be a measurable function on $\Omega$ and $A\subseteq \IR^d$ closed. Write $E_A f$ for the zero extension of $f$ to $A$. Clearly, $E_A$ is isometric from $\L^p(A\cap \Omega)$ to $\L^p(A)$ for all $1\leq p \leq \infty$.
Moreover, if $f \in \C^\infty_D (\Omega) \cap \W^{k,p}(\Omega)$ and $A \cap \overline{\Gamma} = \emptyset$, then $E_A f$ is again in $\C^\infty_D(A) \cap \W^{k,p}(A)$. A relevant example is $A=Q\in \WW_i$. 
Note that then $\| \partial^\alpha E_A f \|_{\L^p (A)} = \| \partial^\alpha f\|_{\L^p (A \cap \Omega)}$ holds for any $|\alpha| \leq k$.

Recall the notation introduced in Remark~\ref{Rem: Adapted Polynomial}. Define the extension operator $E$ on some locally integrable $f$ by
%%RMB added x as an argument to Ef below
\begin{align*}
 E f(x) \coloneqq \begin{cases} f(x), &x \in \Omega,\\ 0, &x \in D,\\ \sum_{Q_j \in \WW_e} (E_{Q_j^*} f)_{Q_j^*}(x) \varphi_j (x), &x \in \overline{\Omega}^c. \end{cases}
\end{align*}
If $\WW_e$ is empty (which is the case if $D = \partial \Omega$ according to Remark~\ref{Rem: Neighborhood of Neumann boundary}) then the sum is empty and its value is considered to be zero.
\begin{remark}
\label{Rem: Definition of extension operator}
If $f$ is locally integrable on $\Omega$ then $E f$ is defined almost everywhere on $\IR^d$ according to Lemma~\ref{Lem: Measure of Neumann boundary}. Moreover, $Ef$ is smooth on $\overline{\Omega}^c$ by construction. Due to Remark~\ref{Rem: Adapted Polynomial}, $E$ restricts to a bounded operator from $\L^p(\Omega)$ to $\L^p(\IR^d)$ for all $1\leq p \leq \infty$. If $f \in \C^\infty_D (\Omega) \cap \W^{k,p}(\Omega)$ then $Ef$ vanishes almost everywhere around $D$. Indeed, this follows from the support assumption on $f$ and the fact that $Q_j^\ast$ is close to $D$ if $x$ is close to $D$, see~\eqref{Eq: Estimate distance of reflected cube}. 
\end{remark}

\subsection*{Estimates for the extension operator} We show estimates for the extension operator on different types of cubes. The overlap considerations from before will permit us to sum them up in Proposition~\ref{Prop: Boundedness estimate} to arrive at exterior estimates for the extension operator.

\begin{lemma}
\label{Lem: Estimating differences of averages}
Let $f \in \C^\infty_D (\Omega) \cap \W^{k,p}(\Omega)$, $0\leq \ell \leq k$, $|\alpha| \leq \ell$, and $1\leq p \leq \infty$.
If $S_1 , \dots , S_m$ is a touching chain of Whitney cubes with respect to $\cl{\Gamma}$ whose length is bounded by a constant $M$,
%provided either by Lemma~\ref{Lem: Building chains} or Lemma~\ref{Lem: Building Poincare chains}, 
then
\begin{align*}
\| \partial^\alpha ((E_{S_1} f)_{S_1} - (E_{S_m} f)_{S_m}) \|_{\L^p(S_1)} \lesssim \diam(S_1)^{\ell-|\alpha|} \|\nabla^\ell f\|_{\L^p(\bigcup_{r = 1}^m (2 S_r) \cap \Omega)},
\end{align*}
where the implicit constant does only depend on $d$, $k$, $p$, and $M$. The assertion remains true if the chain consists of cubes in $\Xi$ of fixed size (not necessarily Whitney cubes). In that case, the set $\bigcup_{r = 1}^m (2 S_r) \cap \Omega$ in the $\L^p$-norm on the right-hand side can be replaced by $\bigcup_{r = 1}^m S_r \cap \Omega$.
\end{lemma}

\begin{proof}
We focus on the case of Whitney cubes, the other case is even simpler.

Note first that the sizes of cubes from the chain are pairwise comparable due to the bound on the chain length.
Using Lemma~\ref{Lem: Polynomial norm on different cubes} (observe that the whole chain is contained in a comparably larger cube) we get
\begin{align*}
&\| \partial^\alpha ((E_{S_1} f)_{S_1} - (E_{S_m} f)_{S_m}) \|_{\L^p(S_1)} \\
\lesssim &\sum_{r = 1}^{m - 1} \| \partial^\alpha ((E_{S_r} f)_{S_r} - (E_{S_{r + 1}} f)_{S_{r + 1}}) \|_{\L^p (S_r)} \\
\lesssim &\sum_{r = 1}^{m - 1} \| \partial^\alpha ((E_{S_r} f)_{S_r} - (E_{S_r \cup S_{r + 1}} f)_{S_r}) \|_{\L^p (S_r)} + \| \partial^\alpha ((E_{S_r \cup S_{r + 1}} f)_{S_r}) - (E_{S_{r + 1}} f)_{S_{r + 1}}) \|_{\L^p (S_{r+1})} \\
 \leq &\sum_{r = 1}^{m - 1} \| \partial^\alpha ((E_{S_r} f)_{S_r} - E_{S_r} f) \|_{\L^p(S_r)} + \| \partial^\alpha (E_{S_r} f - (E_{S_r \cup S_{r + 1}} f)_{S_r}) \|_{\L^p (S_r)} \\
 &\qquad+ \| \partial^\alpha ((E_{S_r \cup S_{r + 1}} f)_{S_r} - E_{S_{r + 1}} f) \|_{\L^p (S_{r + 1})} + \| \partial^\alpha (E_{S_{r + 1}} f - (E_{S_{r + 1}} f)_{S_{r + 1}}) \|_{\L^p(S_{r + 1})} \\
 \lesssim &\sum_{r = 1}^{m - 1} \| \partial^\alpha ((E_{S_r} f)_{S_r} - E_{S_r} f) \|_{\L^p(S_r)} + \| \partial^\alpha (E_{S_r \cup S_{r + 1}} f - (E_{S_r \cup S_{r + 1}} f)_{S_r}) \|_{\L^p (S_r \cup S_{r + 1})} \\
 &\qquad+ \| \partial^\alpha ((E_{S_{r + 1}} f)_{S_{r + 1}} - E_{S_{r + 1}} f) \|_{\L^p(S_{r + 1})}.
\end{align*}
By virtue of Lemma~\ref{Lem: Polynomial Poincare}, the first and the last term in the sum on the right-hand side are controlled by $\diam(S_1)^{\ell-|\alpha|} \|\nabla^\ell f\|_{\L^p(S_r \cap \Omega)}$ and $\diam(S_1)^{\ell-|\alpha|} \|\nabla^\ell f\|_{\L^p(S_{r + 1} \cap \Omega)}$.

If $S_r$ and $S_{r+1}$ are of the same size, the second term can be controlled using Corollary~\ref{Cor: Poincare on subcubes}. Otherwise, assume without loss of generality that $\diam(S_{r + 1}) < \diam(S_r)$. Since the cubes are dyadic, it follows that $S_r \cup S_{r + 1} \subseteq 2 S_r$. Moreover,
\begin{align*}
 \dist(2 S_r , \Gamma) \geq \dist(S_r , \Gamma) - \frac{1}{2} \diam(S_r) \geq \frac{1}{2} \diam(S_r).
\end{align*}
So, $E_{2 S_r} f$ is a smooth extension of $E_{S_r \cup S_{r + 1}} f$ to $2 S_r$, in particular $(E_{S_r \cup S_{r + 1}} f)_{S_r} = (E_{2 S_r} f)_{S_r}$. Invoking Corollary~\ref{Cor: Poincare on subcubes} yields
\begin{align*}
&\| \partial^\alpha (E_{S_r \cup S_{r + 1}} f - (E_{S_r \cup S_{r + 1}} f)_{S_r}) \|_{\L^p (S_r \cup S_{r + 1})} \\ 
\leq{} &\| \partial^\alpha (E_{2 S_r} f - (E_{2 S_r} f)_{S_r}) \|_{\L^p (2 S_r)} \lesssim \diam(S_1)^{\ell-|\alpha|} \|\nabla^\ell f\|_{\L^p((2 S_r) \cap \Omega)}. \qedhere
\end{align*}
\end{proof}

\begin{lemma}
\label{Lem: Estimate on exterior cubes}
Let $f \in \C^\infty_D (\Omega) \cap \W^{k,p}(\Omega)$, $0\leq \ell \leq k$, $|\alpha| \leq \ell$, and $1\leq p \leq \infty$. If $Q_j \in \WW_e$, then
\begin{align*}
\| \partial^\alpha E f \|_{\L^p(Q_j)} \lesssim \diam(Q_j)^{\ell-|\alpha|} \| \nabla^\ell f \|_{\L^p(F(Q_j) \cap \Omega)} + \| \partial^\alpha f \|_{\L^p(Q_j^\ast \cap \Omega)}.
\end{align*}
\end{lemma}

\begin{proof}
Observe that $\varphi_k$ vanishes on $Q_j$ if $Q_k\cap Q_j = \emptyset$. Hence, by definition it holds $E f = \sum_{\substack{Q_k\in \WW_e\\Q_j \cap Q_k \neq \emptyset}} (E_{Q_k^*} f)_{Q_k^*} \varphi_k$ and $\sum_{\substack{Q_k\in \WW_e\\Q_j \cap Q_k \neq \emptyset}} \varphi_k \equiv 1$ on $Q_j$. Consequently, using the Leibniz rule we get
\begin{align*}
 & \| \partial^\alpha Ef \|_{\L^p(Q_j)} \\
 \leq{} & \Big\| \sum_{\substack{Q_k\in \WW_e\\Q_j \cap Q_k \neq \emptyset}} \sum_{\beta \leq \alpha} c_{\alpha,\beta} \partial^{\alpha-\beta} \big[ (E_{Q_k^*} f)_{Q_k^*} - (E_{Q_j^*} f)_{Q_j^*} \big] \partial^\beta \varphi_k \Big\|_{\L^p(Q_j)} + \|\partial^\alpha (E_{Q_j^\ast} f)_{Q_j^\ast} \|_{\L^p(Q_j)} \\
 =:{} &\mathrm{I} + \mathrm{II}.
\end{align*}
We employ the estimate for $\partial^\beta \varphi_k$ and Lemma~\ref{Lem: Polynomial norm on different cubes} (taking Lemma~\ref{Lem: Reflection of cubes} into account), followed by Lemma~\ref{Lem: Estimating differences of averages} and~\eqref{Enum: Bound Intersecting cubes} to derive
\begin{align*}
 \mathrm{I} &\lesssim \sum_{\substack{Q_k\in \WW_e\\Q_j \cap Q_k \neq \emptyset}} \sum_{\beta \leq \alpha} \diam(Q_k)^{-|\beta|} \| \partial^{\alpha-\beta} \bigl[ (E_{Q_k^*} f)_{Q_k^*} - (E_{Q_j^*} f)_{Q_j^*} \bigr] \|_{\L^p(Q_j^*)} \\
 &\lesssim \diam(Q_j)^{\ell-|\alpha|} \| \nabla^\ell f \|_{\L^p(F(Q_j) \cap \Omega)}.
\end{align*}
The term $\mathrm{II}$ is controlled by $\| \partial^\alpha f \|_{\L^p(Q_j^\ast \cap \Omega)}$ using~\eqref{Eq: Polynomial Proj Estimate} from Lemma~\ref{Lem: Polynomial Poincare}; Note that we can switch to $Q_j^\ast$ using Lemma~\ref{Lem: Polynomial norm on different cubes} as in the estimate for term $\mathrm{I}$.
\end{proof}

\begin{lemma}
\label{Lem: Estimate on small adjacent cubes}
Let $f \in \C^\infty_D (\Omega) \cap \W^{k,p}(\Omega)$, $0\leq \ell \leq k$, $|\alpha| \leq \ell$, and $1\leq p\leq \infty$. If $Q \in \WW(\overline{\Omega}) \setminus \WW_e$ intersects a cube in $\WW_e$ and satisfies $\diam(Q) \leq A \delta$, then
\begin{align*}
 \| \partial^\alpha E f \|_{\L^p(Q)} \lesssim \diam(Q)^{\ell-|\alpha|} \| \nabla^\ell f \|_{\L^p(F_P(Q) \cap \Omega)}.
\end{align*}
\end{lemma}

\begin{proof}
Note that $Q$ satisfies the assumptions of Lemma~\ref{Lem: Building Poincare chains}. For $Q_j \in \WW_e$ an intersecting cube of $Q$ let $Q_j^* = S_1 , \dots , S_{m_j}$ be the corresponding touching chain. Then
\begin{align*}
 {}&\| \partial^\alpha Ef \|_{\L^p(Q)} \\
 \lesssim {}&\sum_{\substack{Q_j \in \WW_e\\Q \cap Q_j \neq \emptyset}} \sum_{\beta\leq \alpha} \diam(Q_j^*)^{-|\beta|} \bigl\| \partial^{\alpha-\beta} \bigl[ (E_{Q_j^\ast} f)_{Q_j^\ast} \bigr] \bigr\|_{\L^p(Q_j^\ast)} \\
 \lesssim {}&\sum_{\substack{Q_j \in \WW_e\\Q \cap Q_j \neq \emptyset}} \sum_{\beta\leq \alpha} \diam(Q_j^*)^{-|\beta|} \Big[ \bigl\| \partial^{\alpha-\beta} \bigl[ (E_{S_1} f)_{S_1} - (E_{S_{m_j}} f)_{S_{m_j}} \bigr] \bigr\|_{\L^p(S_1)} \\
 &\qquad \qquad \qquad \qquad+ \| \partial^{\alpha-\beta} (E_{S_{m_j}} f)_{S_{m_j}} \|_{\L^p(S_{m_j})} \Big].
\end{align*}
By virtue of Lemma~\ref{Lem: Estimating differences of averages} and~\eqref{Enum: Bound Intersecting cubes} the first term inside the double sum can be controlled by $\diam(Q)^{\ell-|\alpha|} \| \nabla^\ell f \|_{\L^p(\bigcup_{r = 1}^{m_j} (2 S_r) \cap \Omega)}$. For the second term in the sum, note that $E_{S_{m_j}} f \equiv 0$ on the cube $S_{m_j} \cap Q_j$ and that $\abs{S_{m_j} \cap Q_j} \gtrsim \diam(Q_j)^d$ by Lemma~\ref{Lem: Building Poincare chains}. 
Estimate using Lemma~\ref{Lem: Polynomial norm on different cubes} and the fact that $(E_{S_{m_j}} f)_{S_{m_j}\cap Q_j}$ vanishes that
\begin{align*}
	&\| \partial^{\alpha-\beta} (E_{S_{m_j}} f)_{S_{m_j}} \bigl\|_{\L^p(S_{m_j})}
	\lesssim \| \partial^{\alpha-\beta} \bigl[ (E_{S_{m_j}} f)_{S_{m_j}} - (E_{S_{m_j}} f)_{S_{m_j}\cap Q_j} \bigr] \bigr\|_{\L^p(S_{m_j} \cap Q_j)}.
\end{align*}
Using Lemma~\ref{Lem: Polynomial Poincare} and $\diam(S_{m_j}) \approx \diam(Q_j)$ we further estimate
\begin{align*}
	\leq{} &\bigl\| \partial^{\alpha-\beta} \bigl[ E_{S_{m_j}} f - (E_{S_{m_j}} f)_{S_{m_j}} \bigr] \bigr\|_{\L^p(S_{m_j})} + \bigl\| \partial^{\alpha-\beta} \bigl[ E_{S_{m_j}\cap Q_j} f - (E_{S_{m_j}} f)_{S_{m_j}\cap Q_j} \bigr] \bigr\|_{\L^p(S_{m_j}\cap Q_j)} \\
	\lesssim{} &\diam(S_{m_j})^{\ell-|\alpha|+|\beta|} \| \nabla^\ell f \|_{\L^p(S_{m_j} \cap \Omega)} + \diam(S_{m_j} \cap Q_j)^{\ell-|\alpha|+|\beta|} \| \nabla^\ell f \|_{\L^p(S_{m_j} \cap Q_j \cap \Omega)} \\
	\lesssim{} &\diam(Q_j)^{\ell-|\alpha|+|\beta|} \| \nabla^\ell f \|_{\L^p(S_{m_j} \cap \Omega)}.
\end{align*}
With~\eqref{Enum: Bound Intersecting cubes} and $\diam(Q_j) \approx \diam(Q)$ this concludes the proof.
\end{proof}

\begin{lemma}
\label{Lem: Estimate on large adjacent cubes}
Let $f \in \C^\infty_D (\Omega) \cap \W^{k,p}(\Omega)$, $0\leq \ell \leq k$, $|\alpha| \leq \ell$, and $1\leq p\leq \infty$. If $Q \in \WW(\overline{\Omega}) \setminus \WW_e$ intersects a cube in $\WW_e$ and satisfies $\diam(Q) > A \delta$, then
\begin{align*}
 \| \partial^\alpha E f \|_{\L^p(Q)} \lesssim \max(1, \delta^{-\ell}) \| f \|_{\W^{\ell,p}(\bigcup_{\substack{Q_j \in \WW_e\\Q \cap Q_j \neq \emptyset}} Q_j^* \cap \Omega)}.
\end{align*}
\end{lemma}

\begin{proof}
Note that in fact $\diam(Q) \approx \delta$ because $Q$ intersects $\WW_e$. The same is true for its intersecting Whitney cubes. Hence, with a similar calculation as in Lemma~\ref{Lem: Estimate on exterior cubes} we derive
\begin{align*}
	\| \partial^\alpha Ef \|_{\L^p(Q)} &\lesssim \sum_{\substack{Q_j \in \WW_e \\ Q\cap Q_j \neq \emptyset}} \sum_{\beta\leq \alpha} \delta^{-|\beta|} \| \partial^{\alpha-\beta} (E_{Q_j^\ast} f)_{Q_j^\ast} \|_{\L^p(Q_j^\ast)} \\
	&\lesssim \sum_{\substack{Q_j \in \WW_e \\ Q\cap Q_j \neq \emptyset}} \sum_{\beta\leq \alpha} \delta^{-|\beta|} \| \partial^{\alpha-\beta} f \|_{\L^p(Q_j^\ast \cap \Omega)} \\
	&\lesssim \max(1, \delta^{-\ell}) \| f \|_{\W^{\ell,p}(\bigcup_{\substack{Q_j \in \WW_e\\Q \cap Q_j \neq \emptyset}} Q_j^* \cap \Omega)}. \qedhere
\end{align*}
\end{proof}

\begin{proposition}
\label{Prop: Boundedness estimate}
For all $1 \leq p \leq \infty$ and $0\leq \ell \leq k$ there exists a constant $C > 0$ depending only on $d$, $\eps$, $\delta$, $k$, $p$, $\lambda$, and $K$ such that for all $f \in \C^\infty_D (\Omega) \cap \W^{k,p}(\Omega)$ and $|\alpha| \leq \ell$ one has
\begin{align*}
\| \partial^\alpha E f\|_{\L^p (\overline{\Omega}^c)} \leq C \|f\|_{\W^{\ell, p}(\Omega)}.
\end{align*}
\end{proposition}

\begin{proof}
The estimates for the derivatives in the case $p < \infty$ are deduced by the following calculation based on Lemmas~\ref{Lem: Estimate on exterior cubes},~\ref{Lem: Estimate on small adjacent cubes}, and~\ref{Lem: Estimate on large adjacent cubes}
\begin{align*}
 \|\partial^\alpha E f\|_{\L^p(\overline{\Omega}^c)}^p &= \sum_{Q_j \in \WW_e} \|\partial^\alpha E f\|_{\L^p(Q_j)}^p + \sum_{\substack{Q \in \WW(\overline{\Omega}) \setminus \WW_e \\ Q \cap \WW_e \neq \emptyset}} \|\partial^\alpha E f\|_{\L^p(Q)}^p \\
&\lesssim \sum_{Q_j \in \WW_e} \bigl( \diam(Q_j)^{(\ell-|\alpha|)p} \| \nabla^\ell f \|_{\L^p(F(Q_j) \cap \Omega)}^p + \| \partial^\alpha f \|_{\L^p(Q_j^\ast \cap \Omega)}^p \bigr) \\
&+ \sum_{\substack{Q \in \WW(\overline{\Omega}) \setminus \WW_e \\ Q \cap \WW_e \neq \emptyset \\ \diam(Q) > A \delta}} \max(1, \delta^{-\ell p})\, \| f \|_{\W^{\ell,p}(\bigcup_{\substack{Q_j \in \WW_e\\Q \cap Q_j \neq \emptyset}} Q_j^* \cap \Omega)}^p \\
 &+ \sum_{\substack{Q \in \WW(\overline{\Omega}) \setminus \WW_e \\ Q \cap \WW_e \neq \emptyset \\ \diam(Q) \leq A \delta}} \diam(Q_j)^{(\ell-|\alpha|)p} \| \nabla^\ell f \|_{\L^p(F_P(Q) \cap \Omega)}^p.
\end{align*}
Since $\ell-|\alpha| \geq 0$ and $\diam(Q_j)$ is comparably smaller than $\delta$, we can get rid of the factors in front of the norm terms to the cost of an implicit constant depending only on $\delta$ and $k$. The estimate then follows from Lemma~\ref{Lem: Number of preimages to reflected cube} and from~\eqref{Eq: Nasty chain estimate}, which holds also true for the chains $F_P(Q)$ as mentioned in the discussion previous to~\eqref{Eq: Nasty chain estimate}.

The estimate in the case $p=\infty$ is even simpler because we can use the same estimates but can omit the overlap argument. 
%In the case $p = \infty$, the estimate for the gradient follows by Lemmas~\ref{Lem: Estimate on exterior cubes},~\ref{Lem: Estimate on small adjacent cubes} and~\ref{Lem: Estimate on large adjacent cubes} employed in the calculation
%\begin{align*}
% \|\nabla E f\|_{\L^{\infty}(\overline{\Omega}^c)} &\leq \sup_{Q_j \in \WW_e} \|\nabla E f\|_{\L^{\infty}(Q_j)} + \sup_{\substack{Q \in \WW(\overline{\Omega}) \setminus \WW_e \\ Q \cap \WW_e \neq \emptyset \\ \diam(Q) > A \delta}} \|\nabla E f\|_{\L^{\infty}(Q)} + \sup_{\substack{Q \in \WW(\overline{\Omega}) \setminus \WW_e \\ Q \cap \WW_e \neq \emptyset \\ \diam(Q) \leq A \delta}} \|\nabla E f\|_{\L^{\infty}(Q)} \\
% &\lesssim \sup_{Q_j \in \WW_e} \| \nabla f \|_{\L^{\infty}(F(Q_j) \cap \Omega)} + \delta^{-1} \sup_{\substack{Q \in \WW(\overline{\Omega}) \setminus \WW_e \\ Q \cap \WW_e \neq \emptyset \\ \diam(Q) > A \delta}} \| f \|_{\L^{\infty}(\cup_{\substack{Q_j \in \WW_e\\Q \cap Q_j \neq \emptyset}} Q_j^* \cap \Omega)} \\
% &\qquad+ \sup_{\substack{Q \in \WW(\overline{\Omega}) \setminus \WW_e \\ Q \cap \WW_e \neq \emptyset \\ \diam(Q) \leq A \delta}} \| \nabla f \|_{\L^{\infty}(F_P(Q) \cap \Omega)} \\
% &\lesssim \|f\|_{\L^{\infty} (\Omega)} + \| \nabla f \|_{\L^{\infty} (\Omega)}.
%\end{align*}
%The $\L^p$-estimate follows as well by decomposing $\overline{\Omega}^c$ by means of the Whitney decomposition $\WW(\overline{\Omega})$, followed by an application of H\"older's inequality to the mean values building the extension operator $E$ and an application of Lemma~\ref{Lem: Number of preimages to reflected cube}.
\end{proof}

\begin{remark}
\label{Rem: Homogeneous estimates}
	Assume that $F$ and $G$ are subsets of $\IR^d$ such that the following version of~\eqref{Eq: Nasty chain estimate} holds true:
	\begin{align*}
		\sum_{\substack{Q_j \in \WW_e \\ Q_j \cap F \neq \emptyset}} \chi_{F(Q_j)} (x) \leq C \chi_G(x).
	\end{align*}
	Then we may replace the $\L^p(\cl{\Omega}^c)$ norm on the left-hand side of the estimate in Proposition~\ref{Prop: Boundedness estimate} by an $\L^p(F\cap \cl{\Omega}^c)$ norm and the $\W^{\ell, p}(\Omega)$ norm on the right-hand side by an $\W^{\ell, p}(G\cap \Omega)$ norm. Moreover, if $|\alpha|=\ell$ and $F$ is contained in $N_{A\delta}(\Omega)$, then it suffices to estimate against $\| \nabla^\ell f \|_{\L^p(G \cap \Omega)}$. Indeed, in this case the second term in the final estimate in Proposition~\ref{Prop: Boundedness estimate} vanishes. We will benefit from these observations in Section~\ref{Sec: Homogeneous estimates}.
\end{remark}

\section{Approximation with smooth functions on $\IR^d$}
\label{Sec: Approximation smooth functions}

\noindent In this section, we show that smooth and compactly supported functions on $\IR^d$ whose support stays away from $D$ are dense in $\C^\infty_D(\Omega)\cap \W^{k,p}(\Omega)$. In particular, both classes of functions have the same $\W^{k,p}(\Omega)$-closure. We will benefit from this fact in Section~\ref{Sec: Conclusion of the proof}. To do so, we use an approximation scheme similar to that introduced in~\cite[Sec.~4]{Jones}. The arguments rely on techniques similar to what we have used in the construction of the extension operator.

To begin with, let $f\in \C_D^\infty(\Omega)\cap \W^{k,p}(\Omega)$ and put $\kappa \coloneqq \dist(\supp(f), D)>0$. Furthermore, let $\eta>0$ quantify the approximation error. We need parameters $A$, $B$, $s$, $t$, and $\rho$ for which we will collect several constraints in the course of this section (similar to what we have done in Section~\ref{Sec: Cubes and chains}). Some parameters depend on each other, but there is a non-cyclic order in which they can be picked. This will enable us to show the following proposition.
\begin{proposition}
\label{Prop: Approximation}
	Let $f$, $\eta$, and $\kappa$ be as above. Then there exists a function $g$ which is smooth on $\IR^d$, satisfies $\dist(\supp(g),D) > \kappa/2$, and $\|f-g\|_{\W^{k,p}(\Omega)} \lesssim \eta$. In particular, smooth and compactly supported functions on $\IR^d$  whose support has positive distance to $D$ are dense in $\C_D^\infty(\Omega)\cap \W^{k,p}(\Omega)$ with respect to the $\W^{k,p}(\Omega)$ topology.
\end{proposition}

For brevity, put $\tilde B_{t} \coloneqq N_t(\bd \Omega)$ for the tubular neighborhood of size $t$ around  $\bd\Omega$ and choose $s\in (0,1)$ in such a way that we have the estimate
\begin{align}
\label{eq: Wkp estimate f in boundary strip}
	\|f\|_{\W^{k,p}(\tilde B_{3s} \cap \Omega)} \leq \eta.
\end{align}
We may assume that $s$ is smaller than $\kappa/2$. Furthermore, we define a region near $\Gamma$ that stays away from $D$ and is adapted to the support of $f$, namely
\begin{align*}
	B_t \coloneqq \bigl\{ x\in \IR^d \colon \dist(x, \Gamma) < t \text{ and } \dist(x, D) > \frac{\kappa}{2} \bigr\}.
\end{align*}
Later on, we will only deal with $t\in (0,3s)$, so that~\eqref{eq: Wkp estimate f in boundary strip} will in particular be applicable on $B_t\cap \Omega$.

Denote the zero extension of $f$ to
\begin{align*}
	\Omega_0 \coloneqq \Omega \cup \bigcup_{x\in D} \B(x, 3\kappa/4)
\end{align*}
by $E_0 f$. Note that this function is again smooth since $\dist(\B(x, 3\kappa/4), \supp(f)) \geq \frac{\kappa}{4}$ for $x\in D$.
\begin{lemma}
\label{Lem: balls for mollification}
	Let $x\in \Omega\setminus B_s$, then $\B(x,t) \subseteq \Omega_0$ for all $0<t<s/2$.
\end{lemma}
\begin{proof}
	Recall $s<\kappa/2$. We distinguish two cases.
	
	\noindent\emph{Case 1}: $\dist(x,D) \leq \kappa/2$. Let $z\in D$ with $|x-z|=\dist(x,D)$. For $y\in \B(x,\kappa/4)$ we derive
	\begin{align*}
		|y-z| \leq |x-y| + |x-z| < \kappa/4 + \kappa/2 = 3\kappa/4,
	\end{align*}
	so by choice of $s$ we see $$\B(x,t) \subseteq \B(x, \kappa/4) \subseteq \B(z, 3\kappa/4) \subseteq \Omega_0.$$
	
	\noindent\emph{Case 2}: $\dist(x,D) > \kappa/2$ and consequently $\dist(x,\Gamma) \geq s$. Then $\dist(x,\bd \Omega) \geq \min(\kappa/2,s) = s > t$, therefore $\B(x,t) \subseteq \Omega \subseteq \Omega_0$ (keep in mind $x\in\Omega$).
\end{proof}
\subsection*{A family of interior cubes}
Assume that $\rho$ is a dyadic number and $\mathcal{G}$ is the collection of dyadic cubes of sidelength $\rho$. Recall $\Xi = \IR^d \setminus \cl{\Gamma}$. As before, we write $(\Omega_m)_m$ for the connected components of $\Omega$ whose boundary intersects $\Gamma$ and $(\Upsilon_m)_m$ for the remaining ones. Write $\Sigma'$ for the collection of cubes in $\mathcal{G}$ that are contained in $\Xi$. Moreover, we introduce the collection of cubes
\begin{align*}
	\Sigma \coloneqq \bigl\{ R\in \mathcal{G} \colon \text{there exist } S\in \WW(\overline{\Gamma}) \text{ and } m : \diam(S) \geq A\rho,\; R\subseteq S \;\&\; R\cap \Omega_m \neq \emptyset \bigr\}.
\end{align*}
These cubes take the role of $\WW_i$ in the upcoming approximation construction. Note that $\Sigma \subseteq \Sigma'$. For $R\in \Sigma$ define enlarged cubes
\begin{align*}
	\hat R \coloneqq B\,R \qquad \text{and} \qquad \hat{\hat R} \coloneqq 2B\, R.
\end{align*}
We claim that if we choose $\rho \leq \frac{\kappa}{2\sqrt{d}}$, then $R\subseteq \Omega_0$. Indeed, if $R\cap D = \emptyset$, then $R$ is properly contained in $\Omega$ since it has a non-trivial intersection with $\Omega$ and avoids its boundary. Otherwise, let $z\in R\cap D$, then $R\subseteq \B(z, \diam(R)) \subseteq \B(z, 3\kappa/4)$.

\begin{lemma}
\label{Lem: Sigma covers strip away from boundary}
	There are constants $C_1=C_1(d)>0$ and $C_2=C_2(A,s)>0$ such that
	\begin{align*}
		\bigcup_m \Omega_m \setminus N_s(\Gamma) \subseteq \bigcup_{R\in \Sigma} R,
	\end{align*}
	provided $A \geq C_1$ and $\rho \leq C_2$.
\end{lemma}
\begin{proof}
	Let $x\in \Omega_m \setminus N_s(\Gamma)$. In particular, $x\in \Xi$ and hence there exists $S\in \WW(\cl{\Gamma})$ that contains $x$. Since $\dist(x, \Gamma) \geq s$ by choice of $x$ we conclude $\diam(S) \geq \frac{1}{5} \dist(x, \Gamma) \geq \frac{s}{5}$. Hence, if we choose $\rho \leq \frac{s}{5A}$, then $\diam(S) \geq A\rho$. Let $R$ be some cube in $\mathcal{G}$ that contains $x$. If we demand $A\geq \sqrt{d}$, then $R\subseteq S$ because both are dyadic cubes and they have a common point. Finally, $R\cap \Omega_m \neq \emptyset$ since $x\in \Omega_m$, so $R\in \Sigma$.
\end{proof}
If we do not allow to keep some distance to $\Gamma$, then at least the enlarged cubes $\hat R$ cover the whole Neumann boundary region.
\begin{lemma}
\label{Lem: Enlarged cubes cover whole boundary strip}
	There are constants $C_1=C_1(A,\eps)>0$ and $C_2=C_2(A,\delta,\eps,\kappa,\lambda)>0$ such that
	\begin{align*}
		B_{2s} \; \cap \; \bigcup_m \Omega_m \subseteq \bigcup_{R\in \Sigma} \hat R,
	\end{align*}
	provided $B \geq C_1$ and $\rho \leq C_2$.
\end{lemma}
\begin{proof}
	Let $x\in B_{2s} \cap \Omega_m$. Choose $\rho \leq \frac{\eps}{80A} \min(\delta, \lambda \delta)$. Then $\frac{20A}{\eps}\rho < \lambda \delta/2 \leq \diam(\Omega_m)/2$ by~\eqref{Eq: Diameter Condition}, hence there exists some $y\in \Omega_m$ satisfying $|x-y|=\frac{20A}{\eps} \rho$. Moreover, since $|x-y| < \delta$, there is a curve $\gamma$ subject to Assumption~\ref{Ass: Epsilon-delta assumption} that connects $x$ and $y$. Let $z\in \gamma$ with $|x-z|=\frac{1}{2}|x-y|$. Then
	\begin{align*}
		\dist(z,\Gamma) \geq \frac{\eps}{2} |y-z| \geq \frac{\eps}{4} |x-y| = 5A\rho.
	\end{align*}
	Since $\gamma$ takes its values in $\Xi$, there exists a cube $S\in \WW(\cl{\Gamma})$ with $z\in S$. We deduce $\diam(S) \geq \frac{1}{5} \dist(z,\Gamma) \geq A \rho$. As in the previous lemma, there is some cube $R\in \mathcal{G}$ that contains $z$ and consequently is a subcube of $S$. To conclude that $R\in \Sigma$ we must ensure that $\gamma$ cannot escape $\Omega_m$. To this end, let us assume that $z \not\in \Omega_m$. Since $x\in \Omega_m$, there would be some $\tilde z\in \gamma$ with $\tilde z \in \bd \Omega_m$. Since $\tilde z\not\in \Gamma$ by definition of $\gamma$, we must have $\tilde z \in D$. Now recall that by definition of $B_{2s}$ it holds $\dist(x,D) > \frac{\kappa}{2}$. On imposing the constraint $\rho \leq \frac{\eps^2 \kappa}{40A}$ we then get the contradiction
	\begin{align*}
		\dist(x,D) \leq |x-\tilde z| \leq \mathrm{length}(\gamma) \leq \frac{20A}{\eps^2} \rho \leq \frac{\kappa}{2} < \dist(x,D).
	\end{align*}
	So, indeed, $z\in \Omega_m$ and therefore $R\in \Sigma$.
	Denote the center of $R$ by $x_R$ and estimate
	\begin{align*}
		|x-x_R|_\infty \leq |x-z| + |x_R-z|_\infty \leq \left( \frac{10 A}{\eps}+\frac{1}{2}\right) \rho.
	\end{align*}
	So, if we choose $B\geq \frac{20 A}{\eps}+1$, then $x\in \hat R$.
\end{proof}
We have already mentioned that the collection $\Sigma$ is a substitute for $\WW_i$, so it is not surprising that we want to connect nearby cubes in $\Sigma$ by a touching chain of cubes (which we allow to be in $\Sigma'$) of bounded length.
\begin{lemma}
\label{Lem: touching chain for Sigma}
	There are constants $C_1=C_1(B,d,\eps)>0$, $C_2=C_2(d,\eps)>0$, and $C_3=C_3(B,d,\delta)>0$ such that any pair of cubes $R,S\in \Sigma$ with $\hat{\hat R}\cap \hat S\neq \emptyset$ can be connected by a touching chain of cubes in $\Sigma'$ whose length is controlled by $C_1$, provided that $A\geq C_2$ and $\rho \leq C_3$.
%The lemma remains valid if we replace $\hat S$ by $S$.
\end{lemma}
\begin{proof}
	By definition of $\Sigma$ we can pick $x\in R\cap \Omega_m$ and $y\in S\cap \Omega_\ell$. By assumption we moreover fix $z\in \hat{\hat R} \cap \hat S$. Let $x_R$, $y_S$ denote the centers of $R$ and $S$, then
	\begin{align} \label{Eq: Dist x and y in chain connect}
		|x-y| &\leq \sqrt{d}\bigl(|x-x_R|_\infty + |x_R-z|_\infty + |y_S-z|_\infty + |y-y_S|_\infty \bigr)
		< \sqrt{d} (1+2B) \rho. 
	\end{align}
	If we choose $\rho \leq \frac{\delta}{\sqrt{d}(1+2B)}$, then $|x-y| < \delta$ and we can connect $x$ and $y$ by a curve $\gamma$ subject to Assumption~\ref{Ass: Epsilon-delta assumption}. Let $z\in \gamma$ and pick $Q\in \mathcal{G}$ such that $z\in Q$. By symmetry we assume without loss of generality that $|x-z| \leq |y-z|$. This implies, in particular, that $|x-y| \leq 2 |y-z|$.
	
	\emph{Case 1}: $|x-z| \leq \frac{4\sqrt{d}}{\eps} \rho$. Then, since $R\in \Sigma$, we find $\tilde Q \in \WW(\cl{\Gamma})$ with $R\subseteq \tilde Q$ and $\diam(\tilde Q) \geq A \rho$. Using $x\in R\subseteq \tilde Q$, it follows
	\begin{align*}
		\dist(x,\Gamma) \geq \dist(\tilde Q, \Gamma) \geq \diam(\tilde Q) \geq A\rho,
	\end{align*}
	consequently
	\begin{align*}
		\dist(Q,\Gamma) \geq \dist(x,\Gamma)-|x-z|-\diam(Q) \geq \Bigl(A - \frac{4\sqrt{d}}{\eps} - \sqrt{d}\Bigr) \rho.
	\end{align*}
	We choose $A \geq \sqrt{d}(4/\eps+2)$ to conclude $\dist(Q,\Gamma) \geq \diam(Q)$, in particular $Q \in \Sigma'$.
	
	\emph{Case 2}: $|x-z| > \frac{4\sqrt{d}}{\eps} \rho$. We calculate using~\eqref{Eq: Carrot condition}
	\begin{align*}
		\dist(Q,\Gamma) \geq \dist(z,\Gamma) - \diam(Q) \geq \frac{\eps}{2} |x-z| - \sqrt{d} \rho > \diam(Q).
	\end{align*}
	So, as before, $Q\in \Sigma'$.
	
	Taking~\eqref{Eq: Length condition} and~\eqref{Eq: Dist x and y in chain connect} into account, we get $\mathrm{length}(\gamma) + \diam(Q) \leq \sqrt{d} \bigl(\frac{2B+1}{\eps}+1\bigr) \rho$ and $Q\subseteq \cl{\B}(x,\mathrm{length}(\gamma)+\diam(Q))$. By the usual counting argument that we have already used in Lemma~\ref{Lem: Basic chain lemma} it follows that the number of such cubes $Q$ can be bounded by a constant depending only on $B$, $d$, and $\eps$. We select a touching chain out of that collection of cubes to conclude the proof.% in the case with $\hat S$.
	%The proof for $S$ is the same and it would even be possible to choose $C_3$ slightly larger in that case.
\end{proof}
\begin{remark}
\label{Rem: touching chain for Sigma stays in B3s}
	There is a constant $C=C(B,d,\eps,s)$ such that for $R,S\in \Sigma$ as in the foregoing lemma with $R\cap B_{2s} \neq \emptyset$ we have that the connecting chain stays in $\tilde B_{3s}$ provided $\rho \leq C$. Indeed, let $\tilde C$ be the constant $C_1$ from that lemma with dependence on $B$, $d$, and $\eps$. If $x$ is contained in some cube from the connecting chain between $R$ and $S$ and $y\in R\cap B_{2s}$, then $\dist(x,\bd \Omega) \leq \dist(y, \Gamma) + \tilde C \sqrt{d} \rho < 2s + \tilde C \sqrt{d} \rho$, so the claim follows if we choose $\rho \leq s (\tilde C \sqrt{d})^{-1}$.
\end{remark}
So far, we have seen that near $\Gamma$ and away from $D$ we can reasonably cover the components $\Omega_m$. The next two lemmas show that we will not have to bother with the components $\Upsilon_m$.
\begin{lemma}
\label{Lem: partition of unity stays away from Upsilon}
	There is a constant $C=C(B,d,\delta,\eps,\kappa)>0$ such that for any $R\in \Sigma$ with $\hat{\hat R} \cap B_{2s} \neq \emptyset$ it follows $\hat{\hat R} \cap \bigcup_m \Upsilon_m = \emptyset$ provided that $\rho \leq C$.
\end{lemma}
\begin{proof}
	Assume there exists $y\in \hat{\hat R} \cap \Upsilon_m$. Furthermore, let $x\in R\cap \Omega_\ell$. It holds $|x-y| \leq 2B \sqrt{d} \rho$, so $x$ and $y$ can be connected by a path in $\Xi$ subject to Assumption~\ref{Ass: Epsilon-delta assumption} if we ensure $\rho \leq (4\sqrt{d}B)^{-1}\delta$, and its length can be controlled by $\mathrm{length}(\gamma) \leq \eps^{-1} |x-y|$ according to~\eqref{Eq: Length condition}. Since $x$ and $y$ are in different connected components by assumption, there must be a point $z\in \gamma$ which satisfies $z\in D$. By assumption we may pick some $\tilde z\in \hat{\hat R} \cap B_{2s}$. Then
	\begin{align*}
		\dist(x,D) \geq \dist(\hat{\hat R}, D) \geq \dist(\tilde z, D)-\diam(\hat{\hat R}) > \kappa/2- 2B \sqrt{d} \rho.
	\end{align*}
	On the other hand,
	\begin{align*}
		|x-z| \leq \mathrm{length}(\gamma) \leq \frac{2B \sqrt{d}}{\eps} \rho.
	\end{align*}
	If we choose $\rho \leq \frac{\eps \kappa}{16\sqrt{d} B}$ as well as $\rho \leq \frac{\kappa}{8B\sqrt{d}}$, then we arrive at the contradiction
	\begin{align*}
		\dist(x,D) \leq |x-z| \leq \frac{\kappa}{8} < \frac{\kappa}{4} \leq \dist(x, D). &\qedhere
	\end{align*}
\end{proof}
\begin{lemma}
\label{Lem: Ignore f on Upsilon cap B2s}
	Let $x\in B_{2s} \cap \bigcup_m \Upsilon_m$, then $x\not \in \supp(f)$.
\end{lemma}
\begin{proof}
	Let $x\in B_{2s}\cap \Upsilon_m$, then there is $y\in \Gamma$ such that $|x-y| < 2s$. Since $y\not\in \cl{\Upsilon_m}$, there is $z\in \bd \Upsilon_m \subseteq D$ on the connecting line between $x$ and $y$. Thus,
	\begin{align*}
		\dist(x,D)\leq |x-z| \leq |x-y| < 2s < \kappa = \dist(\supp(f),D).
	\end{align*}
	Consequently, $x\not\in \supp(f)$.
\end{proof}

\subsection*{Construction of the approximation and estimates}

Let $\psi$ be a cutoff function valued in $[0,1]$ which is $1$ on $\overline{B_s}$, supported in $N_s(\overline{B_{s}})$, and satisfies $|\partial^\alpha \psi| \lesssim s^{-|\alpha|}$ for $|\alpha|\leq k$. Moreover, fix an enumeration $(R_j)_j$ of $\Sigma$ and let $\varphi_j$ be a partition of unity on $\bigcup_j \hat R_j$ with $\supp(\varphi_j) \subseteq \hat{\hat R}_j$ and $|\partial^\alpha \varphi_j| \lesssim \rho^{-|\alpha|}$. The implicit constants depend on $\alpha$, $d$, and $B$. Note that according to Lemma~\ref{Lem: Enlarged cubes cover whole boundary strip} this partition of unity exists in particular on $B_{2s}\cap \bigcup_m \Omega_m$.

Now we may construct the approximation $g$ of $f$ for Proposition~\ref{Prop: Approximation}. With Lemma~\ref{Lem: balls for mollification} in mind, choose $t\in (0,s/2)$ small enough that
\begin{align}
\label{Eq: g_2 estimate}
	\|f-E_0f \ast \Phi_t\|_{\W^{k,p}(\Omega\setminus \overline{B_s})} \leq \eta s^k,
\end{align}
where $\Phi_t$ is a mollifier function supported in $\B(0,t)$. Recall the notation for adapted polynomials introduced in Remark~\ref{Rem: Adapted Polynomial} and put 
\begin{align*}
	g_1\coloneqq \sum_j (E_0 f)_{R_j} \varphi_j, \qquad g_2\coloneqq E_0f \ast \Phi_t, \qquad \text{and} \qquad g\coloneqq \psi g_1 + (1-\psi) g_2.
\end{align*}
With a further constraint on $\rho$ we see that $g_1$ vanishes near $D$.
\begin{lemma}
\label{Lem: g_1 vanishes near D}
	There exists a constant $C=C(d,\kappa)>0$ such that $\dist(\supp(g_1),D) \geq 3\kappa/4$, provided $\rho \leq C$.
\end{lemma}
\begin{proof}
	Let $x\in \IR^d$ with $\dist(x,D) \leq \frac{3\kappa}{4}$. If $x\in \supp(\varphi_j)$ then fix some $y\in R_j$. We estimate (with $z$ the center of $R_j$)
	\begin{align*}
		\dist(y,D) \leq |y-z| + |x-z| + \dist(x,D) \leq \frac{1}{2}\sqrt{d}\rho + B \sqrt{d} \rho + \frac{3\kappa}{4}. 
	\end{align*}
	Chose $\rho \leq \frac{\kappa}{4 (1+2B)\sqrt{d}}$, then $\dist(y,D) \leq \frac{7}{8}\kappa < \dist(\supp(f),D)$, so $y\not\in \supp(f)$ and $(f)_{R_j} = 0$ by linearity of the projection. But this means $g_1(x) = 0$.
\end{proof}

\begin{proof}[{Proof of Proposition~\ref{Prop: Approximation}}]
	Assume that all constraints on the parameters collected in this section are fulfilled. We split the proof into several steps.
	
	\textbf{Step 1}: \emph{$g$ is well-defined and smooth}. We have already noticed after the definition of $\Sigma$ that we can ensure that all its cubes are contained in $\Omega_0$, so the usage of polynomial approximations is justified and yields the smooth function $g_1$ on $\IR^d$. By definition of the mollification, $g_2$ is a smooth function in $\Omega \setminus \overline{B_s}$. If $x\in \Omega$ with $\dist(x,D) \leq \kappa/2$, then we get as in Lemma~\ref{Lem: balls for mollification} that $\B(x,\kappa/4) \subseteq \B(z, 3\kappa/4) \subseteq \Omega_0$ for some $z\in D$ and $E_0 f$ vanishes on this ball, so by definition of the mollification, $g_2$ vanishes in that neighborhood of $D$. Together with the knowledge on the support of $1-\psi$ we infer that $(1-\psi) g_2$ can be extend by zero to a smooth function on $\IR^d$.
	
	\textbf{Step 2}: \emph{$\dist(\supp(g),D) \geq \kappa/2$}. First, we have $\dist(\supp(g_1),D) \geq 3\kappa/4$ by Lemma~\ref{Lem: g_1 vanishes near D}.
	On the other hand, we have already noticed in Step~1 that $\dist(\supp(g_2), D) \geq \kappa/2$, which in total gives a distance of at least $\kappa/2$ to $D$.
	
	\textbf{Step 3}: \emph{Split up the terms for estimation}. Let $\alpha$ be some multi-index with $|\alpha| \leq k$. Then
	\begin{align*}
		\partial^\alpha (f-g) &= \partial^\alpha (\psi(f-g_1)) + \partial^\alpha ((1-\psi)(f-g_2)) \\
		&= \sum_{\beta \leq \alpha} c_{\alpha,\beta} \left( \partial^{\alpha-\beta} \psi \partial^\beta (f-g_1) + \partial^{\alpha-\beta} (1-\psi) \partial^\beta (f-g_2) \right) \\
		&=: \sum_{\beta \leq \alpha} c_{\alpha,\beta} (\mathrm{I}_{\alpha,\beta} + \mathrm{II}_{\alpha,\beta}).
	\end{align*}
	Clearly, it suffices to estimate for fixed $\alpha$ and $\beta$ the terms $\mathrm I_{\alpha,\beta}$ and $\mathrm{II}_{\alpha,\beta}$ in the $\L^p(\Omega)$-norm. The estimate for $\mathrm{II}_{\alpha,\beta}$ is possible in a uniform manner whereas for $\mathrm I_{\alpha,\beta}$ we will have to carefully consider different relations between $|\alpha|$, $|\beta|$, and $k$.
	
	\textbf{Step 4}: \emph{Estimate of $\mathrm{II}_{\alpha,\beta}$}. Owing to~\eqref{Eq: g_2 estimate}, this term is under control on keeping $|\partial^{\alpha-\beta} (1-\psi)| \lesssim s^{-|\alpha-\beta|} \leq s^{-k}$ in mind (recall $s<1$).
	
	\textbf{Step 5}: \emph{Reduction of the area of integration in $\mathrm I_{\alpha,\beta}$}. Since the support of $\psi$ is contained in $N_s(\overline{B_s})$, we only have to consider this region. Assume $x\in N_s(\overline{B_s}) \setminus B_{2s}$. Then we must have $\dist(x,D) \leq \kappa/2$. But in this region $f$ and $g_1$ vanish according to the definition of $\kappa$ and Step~2. So we only have to deal with $B_{2s}$. Furthermore, $f$ vanishes on $B_{2s} \cap \bigcup_m \Upsilon_m$ according to Lemma~\ref{Lem: Ignore f on Upsilon cap B2s} and the same is true for $g_1$ owing to Lemma~\ref{Lem: partition of unity stays away from Upsilon}. So in summary, we only need to estimate the term $\mathrm I_{\alpha,\beta}$ on $B_{2s} \cap \bigcup_m \Omega_m$.
	
	\textbf{Step 6}: \emph{Estimate of $\mathrm I_{\alpha,\beta}$ if $|\beta|<|\alpha|$}. Since $\psi = 1$ on $B_s$ and $|\alpha-\beta| \neq 0$, we even only have to estimate the $\L^p$ norm over $(B_{2s}\setminus B_s) \cap \bigcup_m \Omega_m$. Write $M$ for this set. The fact $(B_{2s}\setminus B_s) \cap N_s(\Gamma) = \emptyset$ allows us to use Lemma~\ref{Lem: Sigma covers strip away from boundary} to cover $M$ by cubes from $\Sigma$ to calculate
	\begin{align*}
		\| \partial^{\alpha-\beta} \psi \partial^\beta (f-g_1) \|_{\L^p(M)}^p &\leq \sum_{\substack{R_j\in \Sigma \\ R_j \cap B_{2s} \neq \emptyset}} s^{p(|\beta|-|\alpha|)} \|\partial^\beta \bigl(f-\sum_{\substack{R_k\in \Sigma \\ \hat{\hat R}_k \cap R_j \neq \emptyset}} (E_0 f)_{R_k} \varphi_k \bigr) \|_{\L^p(R_j)}^p.
	\end{align*}
	Using that $(\varphi_k)_k$ is a partition of unity on $R_j$, we derive using the Leibniz rule that on $R_j$ we have
	\begin{align*}
		\partial^\beta \sum_{\substack{R_k\in \Sigma \\ \hat{\hat R}_k \cap R_j \neq \emptyset}} (E_0 f)_{R_k} \varphi_k = \partial^\beta (E_0 f)_{R_j} + \partial^\beta  \sum_{\substack{R_k\in \Sigma \\ \hat{\hat R}_k \cap R_j \neq \emptyset}} \bigl[ (E_0 f)_{R_k} - (E_0 f)_{R_j} \bigr] \varphi_k.
	\end{align*}
	Using Lemma~\ref{Lem: Polynomial Poincare} we can estimate the norm of $\partial^\beta \bigl[ f-(E_0 f)_{R_j} \bigr]$ against $\rho^{k-|\beta|} \| \nabla^k f \|_{\L^p(R_j)}$. From $\rho \leq s \leq 1$ we obtain $s^{|\beta|-|\alpha|} \rho^{k-|\beta|} \leq 1$, so we infer with~\eqref{eq: Wkp estimate f in boundary strip} that
	\begin{align*}
		\sum_{\substack{R_j\in \Sigma \\ R_j \cap B_{2s} \neq \emptyset}} s^{p(|\beta|-|\alpha|)} \| \partial^\beta \bigl[ f-(E_0 f)_{R_j} \bigr] \|_{\L^p(R_j)}^p \lesssim \| \nabla^k f \|_{\L^p(\tilde B_{3s} \cap \Omega)}^p \leq \eta^p.
	\end{align*}
	For the second term, we first expand using the Leibniz rule to obtain 
	\begin{align*}
		\partial^\beta  \sum_{\substack{R_k\in \Sigma \\ \hat{\hat R}_k \cap R_j \neq \emptyset}} \bigl[ (E_0 f)_{R_k} - (E_0 f)_{R_j} \bigr] \varphi_k = \sum_{\substack{R_k\in \Sigma \\ \hat{\hat R}_k \cap R_j \neq \emptyset}} \sum_{\gamma \leq \beta} c_{\beta,\gamma} \partial^{\beta-\gamma} \bigl[ (E_0 f)_{R_k} - (E_0 f)_{R_j} \bigr] \partial^\gamma \varphi_k.
	\end{align*}
	According to Lemma~\ref{Lem: touching chain for Sigma} we can apply Lemma~\ref{Lem: Estimating differences of averages} to the effect that
	\begin{align*}
		\| \partial^{\beta-\gamma} \bigl[ (E_0 f)_{R_k} - (E_0 f)_{R_j} \bigr] \|_{\L^p(R_j)} \lesssim \rho^{k-|\beta|+|\gamma|} \|\nabla^k f \|_{\L^p(G_{j,k})},
	\end{align*}
	where $G_{j,k}$ denotes the connecting chain from Lemma~\ref{Lem: touching chain for Sigma} between $R_j$ and $R_k$. The $\rho$ factor compensates for $s^{|\beta|-|\alpha|}$ and $|\partial^\gamma \varphi_k|$ as before. The sums in $k$ and $j$ add up by similar (but simpler) overlap considerations as already seen in Section~\ref{Sec: The extension operator} for $F_{j,k}$. Finally, since $G_{j,k}$ stays in $\tilde B_{3s}$ by Remark~\ref{Rem: touching chain for Sigma stays in B3s}, we get an estimate against $\eta$ as was the case for the first term.
	
	\textbf{Step 7}: \emph{Estimate of $\mathrm I_{\alpha,\beta}$ if $|\beta|=|\alpha|$}. The estimate follows the same ideas as in Step~6, so we only mention which modifications are needed. 
	
	First of all, we have to estimate over the whole $B_{2s} \cap \bigcup_m \Omega_m$. According to Lemma~\ref{Lem: Enlarged cubes cover whole boundary strip}, this set can be covered by the enlarged cubes $\hat{R}_j$. As there are no derivatives on $\psi$, this term can be ignored. For the $\L^p(\hat{R}_j)$ norm of 
	\begin{align*}
		\partial^\beta  \sum_{\substack{R_k\in \Sigma \\ \hat{\hat R}_k \cap R_j \neq \emptyset}} \bigl[ (E_0 f)_{R_k} - (E_0 f)_{R_j} \bigr] \varphi_k
	\end{align*}
	we use Lemma~\ref{Lem: Polynomial norm on different cubes} to estimate
	\begin{align*}
		\| \partial^{\beta-\gamma} \bigl[ (E_0 f)_{R_k} - (E_0 f)_{R_j} \bigr] \|_{\L^p(\hat{R}_j)} \lesssim \| \partial^{\beta-\gamma} \bigl[ (E_0 f)_{R_k} - (E_0 f)_{R_j} \bigr] \|_{\L^p(R_j)},
	\end{align*}
	where the implicit constant introduces a dependence on $B$ (which determines $\kappa$ in that lemma). Then this term can be handled as in Step~6.
	
	For the term $\partial^\beta \bigl[ f-(E_0 f)_{R_j} \bigr]$ we crudely apply the triangle inequality. Then we can estimate $\partial^\beta f$ directly with~\eqref{eq: Wkp estimate f in boundary strip}, and for $\partial^\beta (E_0 f)_{R_j}$ we estimate with Lemma~\ref{Lem: Polynomial norm on different cubes} and Lemma~\ref{Lem: Polynomial Poincare} that
	\begin{align*}
		\| \partial^\beta (E_0 f)_{R_j} \|_{\L^p(\hat{R}_j)} \lesssim \| \partial^\beta (E_0 f)_{R_j} \|_{\L^p(R_j)} \lesssim \| \nabla^k f \|_{\L^p(R_j)}.
	\end{align*}
	
	\textbf{Step 8}: \emph{Approximation by compactly supported functions}. As we have seen in the previous steps, $g$ is an approximation to $f$ that satisfies all properties but the compact support. But if we multiply $g$ with a cutoff $\psi_n$ from $\B(0,n)$ to $\B(0,2n)$ then this sequence does the job. 
\end{proof}

\section{Conclusion of the proof of Theorem~\ref{Thm: Extension theorem for large radius}}
\label{Sec: Conclusion of the proof}

\noindent First, we show that the extension of a compactly supported function in $\C_D^\infty(\IR^d)\cap \W^{k,p}(\Omega)$ constructed in Section~\ref{Sec: The extension operator} is weakly differentiable up to order $k$. More precisely, we show this for the larger class $\C_D^\infty(\IR^d)\cap \W^{k,\infty}(\Omega)$, which makes this result also applicable for Section~\ref{Sec: First order}. Clearly, compactly supported functions in $\C_D^\infty(\IR^d)\cap \W^{k,p}(\Omega)$ belong to this class, though the inclusion is not topological. Combined with the exterior estimates from Proposition~\ref{Prop: Boundedness estimate} and the density result from Section~\ref{Sec: Approximation smooth functions}, this allows us to conclude Theorem~\ref{Thm: Extension theorem for large radius}.

\begin{proposition}
\label{Prop: Lipschitz continuous  representative}
Let $f \in \C_D^\infty (\IR^d) \cap \W^{k,\infty}(\Omega)$ and $|\alpha|\leq k-1$, then $\partial^\alpha Ef$ exists and has a Lipschitz continuous representative $g_\alpha$ which satisfies $\dist(\supp(g_\alpha), D) > 0$.
\end{proposition}

\begin{proof}
Fix an extension $F \in \C_D^\infty(\IR^d)$ of $f$. We show the claim by induction over $|\alpha|$. By Proposition~\ref{Prop: Boundedness estimate}, $Ef$ is well-defined and bounded. Now assume that $|\alpha|<k$ and $\partial^\alpha Ef$ is well-defined and bounded. It suffices to show that $\partial^\alpha Ef$ is given by a Lipschitz function. To this end, define $g_\alpha$ to equal $\partial^\alpha F$ on $\overline{\Omega}$ and $\partial^\alpha Ef$ otherwise. We proceed in two steps.

\textbf{Step 1}: $g_\alpha$ is a representative of $\partial^\alpha Ef$. That $g_\alpha$ and $\partial^\alpha Ef$ coincide on $\Omega \cup \overline{\Omega}^c$ is by definition. It follows from Remark~\ref{Rem: Definition of extension operator} that $\partial^\alpha Ef$ vanishes on $D$. The same is true for $F$ by assumption. Consequently, Lemma~\ref{Lem: Measure of Neumann boundary} reveals that $g_\alpha$ is a representative of $\partial^\alpha Ef$.

\textbf{Step 2}: $g_\alpha$ is Lipschitz continuous. By assumption, $g_\alpha$ is Lipschitz on $\overline{\Omega}$. Furthermore, $g_\alpha$ is smooth on $\overline{\Omega}^c$ and its gradient is bounded according to Proposition~\ref{Prop: Boundedness estimate}. Hence, $g_\alpha$ is Lipschitz on any line segment contained in the exterior of $\Omega$. The claim follows if we show that $g_\alpha$ is continuous on $\bd \Omega$. This is already established around $D$, so it only remains to show continuity in $x\in \Gamma$ with $\dist(x,D)>0$.

Clearly, it suffices to consider $y\in \overline{\Omega}^c$ close to $x$ to show continuity. Moreover, using the positive distance of $x$ to $D$, we may assume using Lemma~\ref{Lem: Distance reflected cubes} that $y\in Q_j$ for some cube $Q_j\in \WW_e$ and that $Q_j^\ast \subseteq \Omega$. Write $y^j$ for the center of $Q_j$. Fix some cube $R$ which contains $Q_j$ and $Q_j^\ast$ with size comparable to $Q_j^\ast$. Also, note that $Ef(z) = (E_{Q_j^\ast} f)_{Q_j^\ast}(z)$ in a neighborhood of $y^j$ by choice of the partition of unity used in the construction of $E$, and that $E_{Q_j^\ast} f = F$ on $Q_j^\ast$ since $Q_j^\ast$ is properly contained in $\Omega$. Then
\begin{align*}
	|g_\alpha(x)-g_\alpha(y)| &\leq |\partial^\alpha F(x) - \partial^\alpha F(y^j)| + |\partial^\alpha F(y^j) - \partial^\alpha (E_{Q_j^\ast} f)_{Q_j^\ast}(y^j)| + |\partial^\alpha Ef(y^j) - \partial^\alpha Ef(y)| \\
	&\leq \|\partial^\alpha F \|_{\Lip(\IR^d)} |x-y^j| + \| \partial^\alpha \bigl(F-(F)_{Q_j^\ast}\bigr) \|_{\L^\infty(R)} + \| \partial^\alpha Ef\|_{\Lip(Q_j)} \diam(Q_j).
\end{align*}
Clearly, the first and the last term tend to zero when $y$ approaches $x$. Finally, we estimate the second term using Corollary~\ref{Cor: Poincare on subcubes} to get decay of order $\diam(R) \approx \diam(Q_j)$. Hence, $g_\alpha$ is indeed continuous in $x$.
\end{proof}

We are now in the position to prove Theorem~\ref{Thm: Extension theorem for large radius}.

\begin{proof}[Proof of Theorem~\ref{Thm: Extension theorem for large radius}]

Let $f \in \C_D^\infty(\IR^d) \cap \W^{k , p} (\Omega)$ with compact support. First, we treat the trivial case $\overline{\Omega} = \IR^d$. In this situation, we extend $f$ to $D$ by zero. This is a representative according to Lemma~\ref{Lem: Measure of Neumann boundary}, it is weakly differentiable of all orders by assumption on $f$, and the extension is isometric with respect to the norm of $\W^{k,p}(\Omega)$. Hence, this case can be completed by continuity, compare with the conclusion of the general case below. 

Otherwise, derive from Proposition~\ref{Prop: Lipschitz continuous  representative} that $Ef$ has weak derivates up to order $k$ and satisfies $\dist(\supp(E f) , D) > 0$.
From the latter follows in particular that $\partial^\alpha Ef$  vanishes in $D$. Proposition~\ref{Prop: Boundedness estimate} yields the desired estimate on $\IR^d \setminus \overline{\Omega}$. Taking Lemma~\ref{Lem: Measure of Neumann boundary} into account, these estimates sum up to an estimate that holds almost everywhere on $\IR^d \setminus \Omega$, which completes the boundedness assertion.

Because we have the positive distance of the support of $E f$ to $D$, a convolution argument shows that moreover $Ef \in \W^{k,p}_D(\IR^d)$. Finally, we can extend $E$ by density to $\W^{k,p}_D(\Omega)$ using the definition of that space and the density of $\C_D^\infty(\IR^d) \cap \W^{k , p} (\Omega)$ shown in Section~\ref{Sec: Approximation smooth functions}.
\end{proof}

\section{Some additional first-order results}
\label{Sec: First order}

\subsection{Extension of Lipschitz functions vanishing on \texorpdfstring{\boldsymbol{$D$}}{D}}

\begin{definition}
\label{Def: LipDOmega}
	Let $\Omega \subseteq \IR^d$ be open and let $D \subseteq \overline{\Omega}$ be closed. The space of Lipschitz continuous functions that vanish on $D$ is given by
\begin{align*}
	\Lip_D (\Omega) \coloneqq \{u : \overline{\Omega} \to \IR : u \text{ Lipschitz and } u = 0 \text{ on } D\}
\end{align*}
with norm $$\|u\|_{\Lip_D(\Omega)} \coloneqq \max(\|u\|_{\L^{\infty} (\Omega)}, |u|_{\Lip (\Omega)}).$$ Here, $|u|_{\Lip(\Omega)}$ is defined as
\begin{align*}
 \lvert u \rvert_{\Lip (\Omega)} \coloneqq \sup_{\substack{x , y \in \Omega \\ x \neq y}} \frac{\lvert u (x) - u(y) \rvert}{\lvert x - y \rvert}.
\end{align*}
\end{definition}
The following approximation lemma for functions in $\mathrm{Lip}_D (\Omega)$ is a modified version of an argument of Stein~\cite[p.~188]{Stein} and is used as a substitute for the result from Section~\ref{Sec: Approximation smooth functions} in the case $p<\infty$.
\begin{lemma}
\label{Lem: Stein approximation}
Let $f \in \mathrm{Lip}_D (\Omega)$. Then there exists a bounded sequence $(f_n)_n \subseteq \C_D^{\infty}(\IR^d) \cap \W^{1 , \infty}(\Omega)$ that converges to $f$ in $\L^\infty(\Omega)$ and satisfies the estimate $\| f_n \|_{\Lip(\Omega)} \lesssim \| f\|_{\Lip(\Omega)}$, where the implicit constant only depends on $d$.
\end{lemma}

\begin{proof}
It suffices to show the claim for functions defined on $\IR^d$ since by Whitney's extension theorem~\cite[Thm.~3.1.1]{Fine-Properties-Of-Functions} there exists an extension $F\in \Lip_D(\IR^d)$ of $f$ that satisfies $\| F \|_{\Lip(\IR^d)} \lesssim \|f\|_{\Lip(\Omega)}$, where the implicit constant depends only on the dimension $d$. For convenience, we drop $\IR^d$ in the notation of function spaces for the rest of this proof.

Pick a family of functions $\varphi_n: [0,\infty) \to [0,1]$ satisfying for $y \geq x > 0$
\begin{align*}
	\mathrm{(i)} \quad \varphi_n=0 \text{ on } [0,1/n) \qquad  \mathrm{(ii)} \quad \varphi_n=1 \text{ on } (2/n,\infty) \qquad \mathrm{(iii)} \quad |\varphi_n(x)-\varphi_n(y)| \lesssim \frac{1}{x} |x-y|,
\end{align*}
for an explicit construction see~\cite[Thm.~3.7]{Haller-Dintelmann_Knees_Rehberg}. Put $\psi_n(x) \coloneqq \varphi_n(\dist_D(x))$. By construction, $\psi_n$ vanishes around $D$ and, by Lipschitz continuity of the distance function, (iii) yields for $x,y\in \IR^d$ with $\dist_D(x) \leq \dist_D(y)$
\begin{align}
\label{Eq: Lipschitz estimate for psi_n}
	|\psi_n(x)-\psi_n(y)| \lesssim \dist_D(x)^{-1} |x-y|.
\end{align}

It suffices to show that there is a sequence of Lipschitz functions whose supports have positive distance to $D$ which fulfill all claims but smoothness, since then we can conclude using mollification. Note that the mollified sequence converges in $\L^\infty$ because we have Lipschitz continuity.

In this light, define the sequence of functions $f_n \coloneqq \psi_n f$. Clearly, these functions are Lipschitz, and their supports stay away from $D$ because $\psi_n$ has this property. Next, we show that $f_n$ converges to $f$ in $\L^\infty$. To this end, let $x\in \IR^d$ and pick $z\in D$ satisfying $|x-z| = \dist_D(x)$. Since $f(z)=0$, we get
\begin{align*}
	|f(x)-f_n(x)| = (1-\psi_n(x)) |f(x)-f(z)| \leq \|f\|_\Lip (1-\psi_n(x)) \dist_D(x).
\end{align*}
By definition of $\psi_n$, $(1-\psi_n(x)) \dist_D(x) \leq 2/n$. Consequently, $|f(x)-f_n(x)| \to 0$ uniformly in $x$.

It remains to show that the Lipschitz seminorms of $f_n$ can be estimated against $\|f\|_\Lip$. The argument uses the same trick using an element from $D$ as we have just seen. So, let $x,y\in \IR^d \setminus D$. Assume without loss of generality that $\dist_D(x) \leq \dist_D(y)$ and let $z$ realize the distance from $x$ to $D$. Using~\eqref{Eq: Lipschitz estimate for psi_n}, we obtain
\begin{align*}
	|f_n(x)-f_n(y)| &\leq |f(x)-f(y)| \psi_n(y) + |f(x)| |\psi_n(x)-\psi_n(y)| \\ 
	&\lesssim \|f\|_\Lip |x-y| + |f(x)-f(z)| \dist_D(x)^{-1} |x-y|.
\end{align*}
The first term is fine and for the second we notice that
\begin{align*}
	|f(x)-f(z)| \leq \|f\|_\Lip |x-z| = \|f\|_\Lip \dist_D(x). &\qedhere
\end{align*}
\end{proof}

\begin{theorem}
\label{Thm: Lipschitz extension}
	Let $\Omega \subseteq \IR^d$ be an open set and $D \subseteq \overline{\Omega}$ be closed such that $\Omega$ and $D$ are subject to Assumption~\ref{Ass: Epsilon-delta assumption}. Then there exists an extension operator $E$ which is bounded from $\Lip_D(\Omega)$ to $\Lip_D(\IR^d)$.
\end{theorem}
\begin{proof}
Let $f\in \Lip_D(\Omega)$ and let $(\varphi_n)_n$ be the approximation in $\C_D^\infty(\IR^d)\cap \W^{1,\infty}(\Omega)$ constructed in Lemma~\ref{Lem: Stein approximation}. Write $E$ for the extension operator constructed in Section~\ref{Sec: The extension operator} for the case $k=1$. According to Proposition~\ref{Prop: Boundedness estimate} we have $\L^\infty$ bounds for $E$ on $\varphi_n$. In particular, this shows the $\L^\infty(\IR^d)$ bound for $E$ on $f$. Moreover, this permits us to calculate for almost every $x , y \in \IR^d$ that
\begin{align*}
 |E f (x) - E f (y)| = \lim_{n \to \infty} |E \varphi_n (x) - E \varphi_n (y)|.
\end{align*}
By Proposition~\ref{Prop: Lipschitz continuous  representative}, $E \varphi_n$ is Lipschitz and hence 
\begin{align*}
 \lim_{n \to \infty} |E \varphi_n (x) - E \varphi_n (y)| &\leq \liminf_{n \to \infty} \|\nabla E \varphi_n\|_{\L^{\infty} (\IR^d)} |x-y|.
\end{align*}
Proceeding by Proposition~\ref{Prop: Boundedness estimate} and Lemma~\ref{Lem: Stein approximation}, we obtain
\begin{align*}
 \liminf_{n \to \infty} \|\nabla E \varphi_n\|_{\L^{\infty} (\IR^d)} \lesssim \liminf_{n \to \infty} \|\varphi_n\|_{\W^{1 , \infty} (\Omega)} \lesssim \| f \|_{\Lip(\Omega)}.
\end{align*}
So, $Ef$ satisfies a Lipschitz estimate against $\|f\|_{\Lip(\Omega)}$ almost everywhere. Hence, $Ef$ possesses a representative which is Lipschitz on $\IR^d$ and satisfies the boundedness estimate.
\end{proof}

\subsection{An extension using reference geometries}

As a corollary, we obtain the existence of an extension operator on even more general but very inexplicit geometries.

\begin{corollary}
\label{Cor: Extension operator}
Let $\Omega \subseteq \IR^d$ be open, $D \subseteq \partial \Omega$ be closed, and define $\Gamma \coloneqq \partial \Omega \setminus D$. Further, assume that there exists a proper open superset $\Omega_{\Gamma} \supset \Omega$ such that $\Gamma$ is contained in $\partial \Omega_{\Gamma}$ and is relatively open with respect to $\partial \Omega_{\Gamma}$. Finally, assume that $\Omega_{\Gamma}$ and $D^{\prime} \coloneqq \partial \Omega_{\Gamma} \setminus \Gamma$ satisfy Assumption~\ref{Ass: Epsilon-delta assumption}. Then there exists an extension operator $E$ that restricts to a bounded operator from $\L^p(\Omega)$ to $\L^p(\IR^d)$ as well as from $\W^{1,p}_D(\Omega)$ to $\W^{1,p}_D(\IR^d)$ in the case $1 \leq p < \infty$, and which restricts to a bounded operator from $\L^\infty(\Omega)$ to $\L^\infty(\IR^d)$ as well as from $\Lip_D(\Omega)$ to $\Lip_D(\IR^d)$.
The operator norms of $E$ only depend on $d$, $p$, $K$, $\eps$, $\delta$, and $\lambda$. Here, the quantities $K$, $\eps$, $\delta$, and $\lambda$ are measured with respect to $\Omega_{\Gamma}$.
\end{corollary}

\begin{proof}
Throughout this proof, let $\eps , \delta > 0$ be the parameters from Assumption~\ref{Ass: Epsilon-delta assumption} with respect to $\Omega_{\Gamma}$ and $D^{\prime}$.

Let $E_0$ be the operator that extends functions by zero from $\Omega$ to $\Omega_\Gamma$, and let $E_{\Gamma}$ denote the extension operator constructed for $\Omega_{\Gamma}$ in Theorem~\ref{Thm: Extension theorem for large radius} for $k=1$. We claim that $E\coloneqq E_\Gamma \circ E_0$ is the desired extension operator. We proceed in several steps.

\textbf{Step 1}: $E$ is $\L^p$ bounded for $1\leq p \leq \infty$. The respective estimate for $E_0$ is clear by construction. The same is true for $E_\Gamma$ by Theorem~\ref{Thm: Extension theorem for large radius} in the case $p<\infty$. Owing to Remark~\ref{Rem: Definition of extension operator}, the $\L^p$-estimates for $E_\Gamma$ also hold in the case $p=\infty$. Hence, the claim follows by composition.

\textbf{Step 2}: $E_0$ maps $\Lip_D(\Omega)$ boundedly into $\Lip_{D'}(\Omega_\Gamma)$. Let $f \in \mathrm{Lip}_D (\Omega)$.

\emph{Claim 1}: $E_0 f$ is Lipschitz continuous on $\Omega_{\Gamma}$. It suffices to consider $x \in \Omega$ and $y \in \Omega_{\Gamma} \setminus \Omega$ with $|x-y| < \delta$. Let $\gamma$ be the path connecting $x$ with $y$ subject to Assumption~\ref{Ass: Epsilon-delta assumption}. By virtue of~\eqref{Eq: Carrot condition} and the intermediate value theorem there exists $z \in \gamma \cap D$. Now, by Lipschitz continuity of $f$, the fact that $f$ vanishes on $D$, and by~\eqref{Eq: Length condition} one estimates
\begin{align}
\label{Eq: E0 Lipschitz estimate}
 \lvert E_0 f (x) - E_0 f (y) \rvert = \lvert f (x) - f(z) \rvert \leq \lvert f \rvert_{\mathrm{Lip} (\Omega)} \lvert x - z \rvert \leq \frac{\lvert f \rvert_{\Lip (\Omega)}}{\eps} \lvert x - y \rvert.
\end{align}

\emph{Claim 2}: $E_0 f$ vanishes on $D'$. Let $x\in D'$. If $x\in \bd \Omega$, then $x \in D$ since $x\not\in \Gamma$ by definition of $D'$. Hence, $E_0 f(x)=f(x)=0$ by choice of $f$. Otherwise, there is a ball $B$ around $x$ that avoids $\Omega$. Choose a sequence $x_n$ in $B\cap \Omega_\Gamma$ that approaches $x$, by construction of $E_0 f$ and continuity shown in Step 1 we conclude that $E_0 f$ vanishes in $x$.

\emph{Claim 3}: $E_0$ is bounded. The crucial estimate was shown in~\eqref{Eq: E0 Lipschitz estimate}.

\textbf{Step 3}: $E_0$ maps $\W^{1,p}_D(\Omega)$ boundedly into $\W^{1,p}_{D'}(\Omega_\Gamma)$ for $1\leq p < \infty$. Let $f \in \W^{1 , p}_D (\Omega)$ and pick an approximating sequence $(f_n)_n \subseteq \C_D^{\infty}(\Omega)\cap \W^{1,p}(\Omega)$, which exists by definition of the space. Since $f$ vanishes around $D$, $E_0 f_n$ is weakly differentiable on $\Omega_\Gamma$ with $\nabla E_0 f_n = E_0 \nabla f_n$ almost everywhere. Therefore,
\begin{align*}
 \| E_0 f_n \|_{\W^{1 , p} (\Omega_{\Gamma})} = \| f_n \|_{\W^{1 , p} (\Omega)},
\end{align*}
which yields that there exists $g \in \W^{1 , p} (\Omega_{\Gamma})$ such that a subsequence $E_0 f_{n_j}$ converges weakly to $g$ in $\W^{1 , p} (\Omega_{\Gamma})$. By the $\L^p$-continuity of $E_0$ we conclude that $E_0 f$ coincides with $g$. Finally, $E_0 f$ belongs to $\W^{1,p}_{D'}(\Omega_\Gamma)$ by construction.
%But since $E_0 f_n$ are continuous functions which vanish identically on $D'$, this follows by an easy approximation argument using the fact that $\W^{1,p}(\Omega_\Gamma)$ is closed under truncations, see~\cite[Sec.~9.2]{Adams-Hedberg}.

\textbf{Step 4}: $E$ is bounded in $\W^{1,p}_D(\Omega)$ for $1\leq p < \infty$ and $\Lip_D(\Omega)$. This follows by composition using Steps~2 or~3 together with Theorem~\ref{Thm: Lipschitz extension} or Theorem~\ref{Thm: Extension theorem for large radius}.
\end{proof}

\begin{remark}
\begin{enumerate}
	\item Notice that Assumption~\ref{Ass: Epsilon-delta assumption} is an explicit assumption that uses only information on points in $\Omega$. To the contrary of that, the geometry described in Corollary~\ref{Cor: Extension operator} has an inexplicit nature, as it is a priori not clear how to construct such a set $\Omega_{\Gamma}$. However, there are important examples where this condition can be checked in the \enquote{blink of an eye}, see Example~\ref{Ex: Blink of an eye} below.
	\item We suggest that a similar result holds in the higher-order case using an induction similar to that in Proposition~\ref{Prop: Lipschitz continuous  representative}. Moreover, a more involved approximation procedure than our truncation method employed at the end of Step~3 would be needed.
\end{enumerate}
\end{remark}

\begin{example}[Exterior boundary cusps at zero or at infinity]
\label{Ex: Blink of an eye}
Let $\Omega$ be a domain that has an exterior boundary cusp either at zero or at infinity, as it is informally depicted in Figure~\ref{Fig: Three examples}. In this case, $\Omega$ is an $\W^{1 , p}_D$-extension domain as a simple reflection argument shows. However, it is not so clear if it satisfies Assumption~\ref{Ass: Epsilon-delta assumption}. Nevertheless, it is simple to verify the validity of the geometric setting stated in Corollary~\ref{Cor: Extension operator}. Indeed, simply take as $\Omega_{\Gamma}$ the lower half-space and notice that the parameter $K$ in Assumption~\ref{Ass: Epsilon-delta assumption} can be set to zero because it is already an $(\eps , \delta)$-domain. \par
We can even go further and extend the geometric setting from Example~\ref{Ex: Sector} to the following one (see Figure~\ref{Fig: Three examples}). \par
Let $\theta \in (0 , \pi)$ and let $\S_{\theta} \subseteq \IR^2$ denote the open sector symmetric about the positive $x$-axis with opening angle $2 \theta$. Define $T_{\theta} \coloneqq \S_{\theta} \cap \{ (x , y) \in \IR^2 : y > 0 \}$. Let $\Omega \subseteq \IR^2$ be a domain satisfying $T_{\theta} \subseteq \Omega^c$ and define
\begin{align*}
 \Gamma \coloneqq (0 , \infty) \times \{ 0 \} \qquad \text{and} \qquad D \coloneqq \partial \Omega \setminus \Gamma.
\end{align*}
Assume further, that $\Omega$ is such that $D$ is closed (this avoids that $D$ touches $\Gamma$ from below). To apply Corollary~\ref{Cor: Extension operator} take $\Omega_{\Gamma} \coloneqq (\overline{T_{\theta}})^c$. As this is an $(\eps , \delta)$-domain, it satisfies Assumption~\ref{Ass: Epsilon-delta assumption} with $K = 0$.
\end{example}

\begin{figure}[!ht]
\centering
\begin{minipage}{.28\textwidth}
  \centering
  \includegraphics[width=0.8\textwidth]{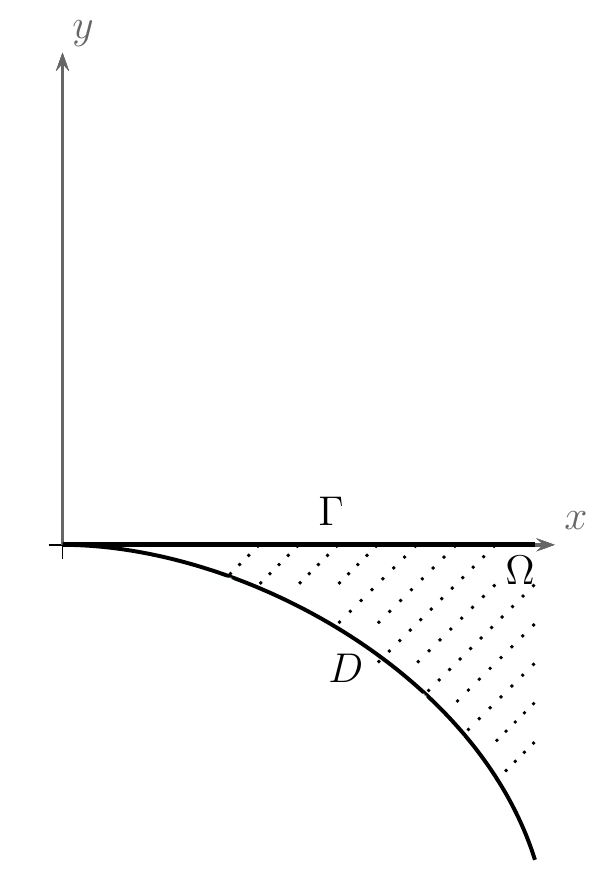}
\end{minipage}
\hfill
\begin{minipage}{.28\textwidth}
  \centering
  \includegraphics[width=0.85\textwidth]{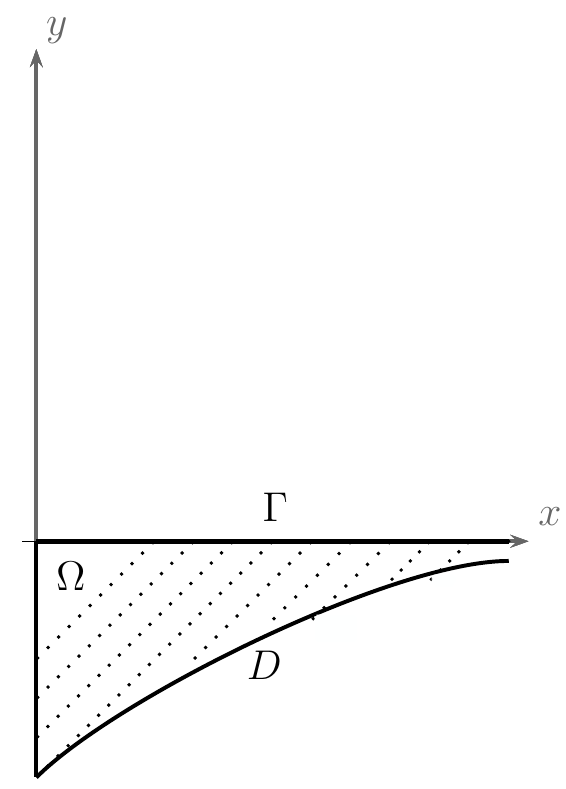}
\end{minipage}
\hfill
\begin{minipage}{.4\textwidth}
  \centering
  \includegraphics[width=1.0\textwidth]{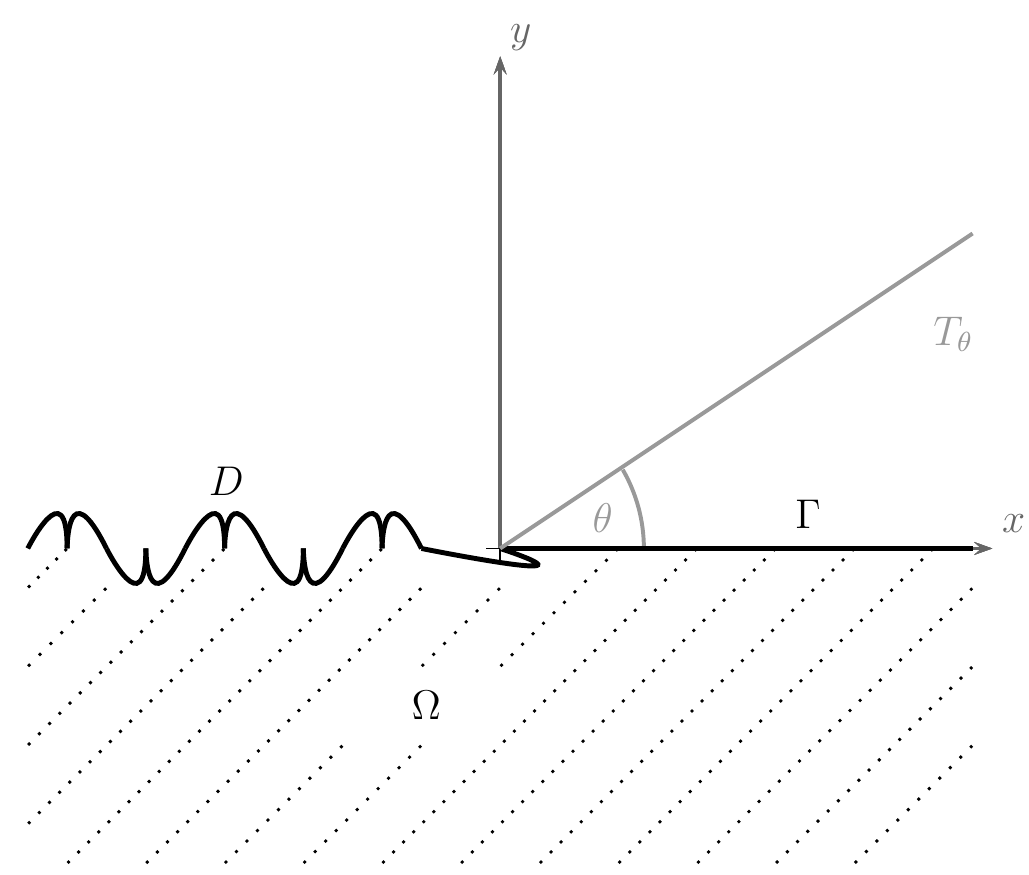}
\end{minipage}
  \caption{Situations in Example~\ref{Ex: Blink of an eye}.}
\label{Fig: Three examples}
\end{figure}

\section{Homogeneous estimates}
\label{Sec: Homogeneous estimates}

\noindent We provide further estimates for the extension operator from Theorem~\ref{Thm: Extension theorem for large radius} which concern homogeneous estimates and locality (see Definition~\ref{Def: locality} for a proper definition). These results build on the observations made in Remark~\ref{Rem: Homogeneous estimates}.
\begin{definition}
\label{Def: locality}
	An extension operator $E$ on $\W^{k,p}_D(\Omega)$ is called \emph{local} if there exist constants $r_0, \kappa>0$ such that $$\| \nabla^\ell E f\|_{\L^p(\B(x,r))} \lesssim \| f \|_{\W^{k,p}(\Omega \cap \B(x, \kappa r))}$$ for all $x\in \bd \Omega$, $r\in (0, r_0)$, and $\ell \leq k$. Moreover, call $E$ \emph{homogeneous} if one can replace the right-hand side of that estimate by $\| \nabla^\ell f \|_{\L^p(\Omega \cap \B(x, \kappa r))}$.
\end{definition}
To verify that $E$ is local, we chose $F=\B(x,r)$ in Remark~\ref{Rem: Homogeneous estimates} and let $Q_j\in \WW_e$ with $Q_j \cap \B(x,r) \neq \emptyset$. On using~\eqref{Eq: Comparison of cube and reflected cube},~\eqref{Eq: Estimate distance of reflected cube}, the bound on the chain length from Lemmas~\ref{Lem: Building chains} as well as the properties of Whitney cubes, we see that $F(Q_j)$ is contained in the ball $\B(x,\kappa r)$ for some $\kappa$ depending only on $\eps$, $d$, $K$, and $\lambda$ (as before, an analogous version for $F_P(Q)$ holds on using Lemma~\ref{Lem: Building Poincare chains} instead of Lemma~\ref{Lem: Building chains} and a similar reasoning). So, with $G=\B(x, \kappa r)$ we derive locality from Remark~\ref{Rem: Homogeneous estimates} with $r_0=\infty$. If we restrict to $r_0=A\delta$, the same remark also yields that $E$ is homogeneous. Note that in the case of $\delta=\infty$ this restriction is void. We summarize this result in the following theorem.
\begin{theorem}
\label{Thm: Local homogeneous estimates}
	Let $\Omega \subseteq \IR^d$ be open and $D \subseteq \partial \Omega$ be closed such that $\Omega$ and $D$ are subject to Assumption~\ref{Ass: Epsilon-delta assumption}, and fix some integer $k \geq 0$. Then there exist $A,\kappa>0$ and an extension operator $E$ such that for all $1 \leq p < \infty$ one has that $E$ restricts to a bounded mapping from $\W^{k,p}_D(\Omega)$ to $\W^{k,p}_D(\IR^d)$ and which is moreover homogeneous and local, that is, the estimate $$\|\nabla^\ell Ef \|_{\L^p(\B(x,r))} \lesssim \| \nabla^\ell f \|_{\L^p(\B(x,\kappa r) \cap \Omega)}$$ holds for $f\in \W^{k,p}_D(\Omega)$, $\ell \leq k$, $x\in \bd \Omega$, and $r\in (0, A\delta)$.
The implicit constant in that estimate depends on $d$, $p$, $K$, $k$, $\eps$, $\delta$, and $\lambda$.
\end{theorem}

%%%%%%%%%%%%%%%%%%%%%%%%%%%%%%%%%%%%%%%%%%%%%%%%%%%%%%%%%%%%%%%%%%%%%%%%%%%%%%%%%%%%%%%%%%%%%%%%%%%%%%%%%%%%%%%%%%%%%%%%%%%%%%%%%%%%%%%%%%%%%%%%%%%%%%%%%%%%%%%%%%%%

\end{document}